\documentclass[11pt,a4paper]{amsart}
\usepackage{a4wide,amsfonts,amsmath,latexsym,amssymb,euscript,eufrak,graphicx,units,mathrsfs,bbold}   

\usepackage{amsmath, amsthm, amsfonts, amssymb,color}

\def\HH{\EuFrak H}
\def\RR{\mathbb R}
\def\R{\mathbb R}
\def\E{\mathbb E}
\def\EE{\mathbb E}
\def\DD{\mathbb D}
\def\Z{\mathbb Z}
 
\def\Var{{\rm Var}}
\def\dom{{\rm Dom}}

\newtheorem{prop}{Proposition}[section]
\newtheorem{proposition}{Proposition}[section]

\newtheorem{lemma}[prop]{Lemma}

\newtheorem{theorem}[prop]{Theorem}

\newtheorem{remark}[prop]{Remark}

\numberwithin{equation}{section}
\begin{document}
   
   \title[Rate of convergence in the BM theorem]{Rate of convergence in the Breuer-Major theorem via chaos expansions}

 \author[S. Kuzgun]{Sefika Kuzgun}
\address{University of Kansas, Department of Mathematics, USA}
\email{sefika.kuzgun@ku.edu}

\author[D. Nualart]{David Nualart} \thanks {D.\ Nualart is supported by the  NSF Grant DMS 1811181.}
\address{University of Kansas, Department of Mathematics, USA}
\email{nualart@ku.edu}

\begin{abstract}
We show new estimates for the total variation and Wasserstein distances in the framework of the
Breuer-Major theorem. The results are based  on the combination of Stein's method for normal approximations and Malliavin calculus together with
Wiener chaos expansions.

\end{abstract}

\maketitle

   \section{Introduction}
 Suppose that $X=\{ X_n , n\ge 0\}$  is a centered stationary Gaussian  sequence of random variables     with unit variance.
  For all $k \in \Z$, set $\rho(k) = \E(X_0 X_k)$ if $k\ge 0$ and $\rho(k) = \rho(-k)$ if $k<0$. 
   We say that a function $g \in L^2( \mathbb{R}, \gamma)$, where $\gamma$ is the standard Gaussian measure, has  
  {\it Hermite rank}   $d\ge 1$ if
	\begin{equation}\label{hexp}
g(x)= \sum_{q=d} ^\infty c_q  H_q(x),
\end{equation}
 where $c_d \not =0$ and $H_q$ is the $q$-th Hermite polynomial. We will  make use of    the following condition that relates the covariance function $\rho$ to the Hermite rank of a function $g$:
 \begin{equation} \label{h1}
\sum_{j \in \Z} |\rho(j)|^d < \infty.
\end{equation}
   The Breuer-Major theorem  (see \cite{bm}) says that, under condition (\ref{h1}), 
    the sequence
	  \begin{equation}\label{yn}
	  F_n := \frac{1}{\sqrt{n}} \sum_{i=1}^n g(X_i) 
	  \end{equation}
  converges in law to the  normal distribution $N(0, \sigma^2)$, where
	\begin{equation}\label{bm.sig}
		\sigma^2 = \sum_{q=d}^\infty q! c_q^2 \sum_{k \in \Z} \rho(k)^q.
	\end{equation}
 
   The aim of this paper is to  estimate the rate of convergence to zero  of the total variation and Wasserstein distances between  the normalized  sequence  
   \begin{equation} \label{yn}
   Y_n:=\frac {F_n } {\sqrt{\Var (F_n)} } 
   \end{equation}
and  the standard normal law $N(0,1)$, assuming minimal regularity and integrability  conditions on the function $g$.
To show these results we will apply a combination of Stein's method for normal approximations and techniques of Malliavin calculus, and we will make use of the Wiener chaos expansion of the random variable $F_n$. The combination of Stein's method with Malliavin calculus to study normal approximations was first developed by Nourdin and Peccati  (see the pioneering work \cite{np-ptrf} and the monograph  \cite{np-book}). For random variables on a fixed Wiener chaos, these techniques provide a quantitative version of the Fourth Moment Theorem proved by Nualart and Peccati in \cite{nunugio}.  
  
Given a function $g \in L^2( \mathbb{R}, \gamma)$ with expansion (\ref{hexp}), we denote by $A(g)$ the function in $ L^2( \mathbb{R}, \gamma)$, whose Hermite coefficients are the absolute values of the coefficients of $g$, that is,
\begin{equation} \label{ma2}
A(g)(x) =\sum_{q=d} ^\infty |c_q| H_q(x).
\end{equation}
For any integer $k\ge 1$ and any real $p\ge 1$, we denote by $\DD^{k,p}(\R,\gamma)$ the Sobolev space of functions which are $k$ times weakly differentiable, such that together with their derivatives  up to order  $k$, they have finite moments of order $p$ with respect to the measure $\gamma$.
Also, we denote by $   d_{\rm TV}$ and $d_{\rm W}$ the total variation and Wasserstein distances, respectively. Along the paper, $Z$ will denote a $N(0,1)$ random variable.
Our first result is the following.

\begin{theorem} \label{thm1}
Assume that $g\in L^2(\RR,\gamma)$ has Hermite rank $d\ge 2$ and  satisfies $A(g)\in \DD^{1,4}(\RR,\gamma)$. Suppose that   (\ref{h1}) holds true and
let $Y_n$ be the random variable defined in (\ref{yn}).  Then we have the following estimates:
\begin{itemize}
\item[(i)] If $d=2$, then
	\begin{align} \label{ecu1}
				   d_{\rm TV}(Y_n , Z)  \leq  C n^{-\frac{1}{2}} \left(\sum_{|k| \leq n} |\rho(k)|\right)^{\frac{1}{2}}  + C n^{-\frac{1}{2}} \left(\sum_{|k| \leq n} |\rho(k)|^{\frac{4}{3}}\right)^{\frac 32 } \,.
		\end{align}
\item[(ii)] If $d\ge 3$, we have
		   \begin{align} 
		   d_{\rm TV} (Y_n ,Z) &\le C n^{-\frac{1}{2}} \sum_{|k| \leq n} |\rho(k)|^{d-1} \left(\sum_{|k| \leq n} |\rho(k)|^{2}\right)^{\frac{1}{2}} \notag\\
		& \quad +    C n^{-\frac{1}{2}} \left(\sum_{|k|\leq n}|\rho(k)|^2\right)^{\frac{1}{2}} \left(\sum_{|k|\leq n}|\rho(k)|\right)^{\frac{1}{2}}\label{ecu2}\,.
\end{align}
\end{itemize}
\end{theorem}

The proof of these results is based on Proposition \ref{prop2}, that requires the estimation of $\Var (\langle DF_n, u_n \rangle_{\HH})$, where $u_n$ is such that $F_n=\delta(u_n)$. Here $D$ and $\delta$ are the derivative and divergence operators associated with the Malliavin calculus for the Gaussian sequence $X$. Following the ideas developed in \cite{NN} and \cite{NZ}, we construct the sequence $u_n$ using the  operator $T_1(g)$ that  shifts in one unit the  Hermite expansion of $g$.  A basic ingredient of the proof is  an explicit computation of the variance $\Var (\langle DF_n, u_n \rangle_{\HH})$, using Wiener chaos expansions. For this we need  a result on the convergence in $L^2$ of    powers of truncated  Wiener chaos  expansions  established in Proposition \ref{prop1}, which has its own interest. 
A sufficient condition for a function $g$ to satisfy  $A(g)\in \DD^{k,M}(\RR,\gamma)$ for any integer $k \ge 0$, $M\geq 3$ is given in  Lemma \ref{lem3.3}.

Let us compare Theorem  \ref{thm1} with the existing results in the literature. For $d=2$, the estimate  (\ref{ecu1}) coincides with 
the estimate obtained in \cite{NZ} (see Theorem 4.3 (iii)), assuming  $g\in \DD^{4,4} ( \R,\gamma)$. This is the best estimate that one can obtain using Proposition \ref{prop2} (it coincides with the bound for $g(x)= x^2-1$). In \cite{NZ} this estimate is obtained applying Poincar\'e inequality to estimate the variance plus twice  the integration-by-parts formula and for this reason one requires the function $g$ to be four times differentiable. Here, we only need one derivative, but for the function $A(g)$.
In a recent note (see \cite{NPY}), the authors have obtained the weaker  bound 
\begin{equation} \label{e1}
     d_{\rm TV}(Y_n, Z)  \leq    C n^{-\frac{1}{2}} \left(\sum_{|k| \leq n} |\rho(k)|\right)^{\frac 32 } 
     \end{equation}
     assuming only  $g\in \DD^{1,4} ( \R,\gamma)$ and applying  Gebelein's inequality, instead of Poincar\'e's inequality, to estimate the variance of   $\langle DF_n, u_n \rangle_{\HH}$.
    Notice that the bound (\ref{e1}) holds, for example, for the function $g(x)= |x| - \E(|Z|)$, which belongs to  $ \DD^{1,4}(\R, \gamma)$.
        
         In the case $d\ge 3$,  the estimate  (\ref{ecu2}) coincides with 
the estimate obtained in \cite[Theorem 4.5]{NZ},   assuming  $g\in \DD^{3d-2,4} ( \R,\gamma)$, and applying the integration-by-parts argument several times.  Again our estimate requires only one derivative (for $A(g)$) instead of $3d-2$ derivatives. Also, computing the third and fourth cumulants in the case $g=H_d$, leads to the optimal bound (see  \cite{bbnp})
    \[
     d_{\rm TV} (Y_n ,Z) \\
		   \le  \frac Cn \left( \sum_{|k| \leq n} |\rho(k)|^{d-1} \right)^2 \sum_{|k| \leq n} |\rho(k)|^2
		   +  \frac{C}{\sqrt{n}} \left( \sum_{|k| \leq n} |\rho(k)|^{\frac {3d} 4} \right)^2 \mathbf{1}_{\{d \,\, {\rm even}\}}.
\]

The second part of the paper is devoted to showing two improvements of the above bound for $d=2$.
First we establish the following upper bound for the Wasserstein distance, using a new estimate (see Proposition \ref{prop3}) and the representation of $F_n$ as an iterated divergence $F_n = \delta^2(v_n)$. 
  \begin{theorem} \label{thm2}
Assume that $g\in L^2(\RR,\gamma)$ has Hermite rank $d= 2$ and  satisfies $A(g)\in \DD^{2,6}(\RR,\gamma)$. Suppose that   (\ref{h1}) holds true and let $Y_n$ be the random variable defined in (\ref{yn}).   Then we have the following estimate
\begin{equation}  \label{equa1}
				   d_{\rm W}(Y_n , Z)  \leq  C n^{-\frac{1}{2}} \left(\sum_{|k| \leq n} |\rho(k)|\right)^{\frac{1}{2}}
				   + C n^{-\frac{1}{2}} \left(\sum_{|k| \leq n} |\rho(k)| ^{\frac 32}\right)^2.
  \end{equation}
  \end{theorem}

Going back to the total variation distance, we recall  first that the  optimal bound for $d=2$ is  
    \begin{equation} \label{optimal}
     d_{\rm TV}(Y_n, Z)  \leq    C n^{-\frac{1}{2}} \left(\sum_{|k| \leq n} |\rho(k)|  ^{\frac 32}\right)^{2 }.
     \end{equation}
     This estimate  was obtained for $g=H_2$ in \cite{np-15}, with a matching lower bound, and it was extended to  $g\in \DD^{6,8}(\R,\gamma)$ in \cite{NZ}. This upper bound, however, cannot be obtained as a consequence of Proposition \ref{prop2} and requires a more intensive application of  Stein's method (see  \cite{np-15,NZ}). Using Proposition \ref{prop2a}, we have obtained the following result.
     \begin{theorem} \label{thm1a}
Assume that $g\in L^2(\RR,\gamma)$ has Hermite rank $d= 2$ and  satisfies $A(g)\in \DD^{3,8}(\RR,\gamma)$. Suppose that   (\ref{h1}) holds true and
let $Y_n$ be the random variable defined in (\ref{yn}).  Then  the  estimate (\ref{optimal}) holds true.
\end{theorem}
     
   Notice that the first term in (\ref{equa1}) coincides with the first term in  (\ref{ecu1}), while the second term is precisely the optimal rate for the total variation distance (\ref{optimal}).
  
  The paper is organized as follows. Section 2 reviews some preliminaries on the Malliavin calculus for an isonormal Gaussian process and  Stein's method. Section 3 presents a new result on the convergence in $L^2(\Omega)$ of  powers of Wiener chaos expansions, which has its own interest. Finally, Sections 4,  5 and 6 contain the proofs of Theorems \ref{thm1}, \ref{thm2} and \ref{thm1a}, respectively.
  
     Along the paper we will  denote by $C$ a generic constant that may vary from line to line.

   \section{Preliminaries}
  
   In this section, we briefly recall some  elements of  the Malliavin calculus associated with a Gaussian family of random variables. 
   We refer the reader to \cite{np-book,nualartbook,CBMS} for a detailed account on this topic.  We will also recall two basic inequalities for the total variation distance proved using  Stein's method and we present a new inequality for the Wasserstein distance.
   
\subsection{Malliavin calculus}

 Let $\mathfrak{H}$ be a real separable Hilbert space. For any integer $m \geq 1$, we use $\mathfrak{H}^{\otimes m}$ and $\mathfrak{H}^{\odot m}$ to denote the $m$-th tensor product and the $m$-th symmetric tensor product of $\mathfrak{H}$, respectively. Let $W = \{W(\phi), \phi \in \mathfrak{H}\}$  denote an isonormal Gaussian process over the Hilbert space $\mathfrak{H}$. That means, $W$ is a centered Gaussian family of random variables, defined on some probability space $(\Omega, \mathcal{F}, P)$, with covariance $$\EE\left(W(\phi)W(\psi)\right) = \langle \phi, \psi \rangle_{\mathfrak{H}}, \qquad \phi, \psi \in \mathfrak{H}.$$
We assume that $\mathcal{F}$ is generated by $W$.

We denote by $\mathcal{H}_m$ the closed linear subspace of $L^2(\Omega)$ generated by the random variables $\{H_m(W(\varphi)): \varphi \in \mathfrak{H}, \|\varphi\|_{\mathfrak{H}}=1\}$, where $H_m$ is the $m$-th Hermite polynomial defined by 
\[
H_m(x)=(-1)^me^{\frac{x^2}{2}}\frac{d^m}{dx^m}e^{-\frac{x^2}{2}},\quad m \geq 1,
\]
and $H_0(x)=1$. The space $\mathcal{H}_m$ is called the Wiener chaos of order $m$. The $m$-th multiple integral of $\phi^{\otimes m} \in \mathfrak{H}^{\odot m}$ is defined by the identity 
$ I_m(\phi^{\otimes m}) = H_m(W(\phi))$ for any $\phi\in \mathfrak{H}$ with $\| \phi\|_{\HH}=1$. The map $I_m$ provides a linear isometry between $\mathfrak{H}^{\odot m}$ (equipped with the norm $\sqrt{m!}\|\cdot\|_{\mathfrak{H}^{\otimes m}}$) and $\mathcal{H}_m$ (equipped with $L^2(\Omega)$ norm). By convention, $\mathcal{H}_0 = \mathbb{R}$ and $I_0(x)=x$.

The space $L^2(\Omega)$ can be decomposed into the infinite orthogonal sum of the spaces $\mathcal{H}_m$. Namely, for any square integrable random variable $F \in L^2(\Omega)$, we have the following expansion,
\begin{equation} \label{chaos}
  F = \sum_{m=0}^{\infty} I_m (f_m) ,
\end{equation}
where $f_0 = \mathbb{E}(F)$, and $f_m \in \mathfrak{H}^{\odot m}$ are uniquely determined by $F$. This is known as the Wiener chaos expansion.  

  For a smooth and cylindrical random variable $F= f(W(\varphi_1), \dots , W(\varphi_n))$, with $\varphi_i \in \mathfrak{H}$ and $f \in C_b^{\infty}(\mathbb{R}^n)$ ($f$ and its partial derivatives are bounded), we define its Malliavin derivative as the $\mathfrak{H}$-valued random variable given by
\[
 DF = \sum_{i=1}^n \frac{\partial f}{\partial x_i} (W(\varphi_1), \dots, W(\varphi_n))\varphi_i\ .
\]
By iteration, we can also define the $k$-th derivative $D^k F$, which is an element in the space $L^2(\Omega; \mathfrak{H}^{\otimes k})$. For any real $p\ge 1$ and any integer $k\ge 1$, the Sobolev space $\mathbb{D}^{k,p}$ is defined as the closure of the space of smooth and cylindrical random variables with respect to the norm $\|\cdot\|_{k,p}$ defined by 
\[
 \|F\|^p_{k,p} = \mathbb{E}(|F|^p) + \sum_{i=1}^k \mathbb{E}(\|D^i F\|^p_{\mathfrak{H}^{\otimes i}}).
\]
We define the divergence operator $\delta$ as the adjoint of the derivative operator $D$. Namely, an element $u \in L^2(\Omega; \mathfrak{H})$ belongs to the domain of $\delta$, denoted by $\dom\, \delta$, if there is a constant $c_u > 0$ depending on $u$ and satisfying 
\[
|\mathbb{E} (\langle DF, u \rangle_{\mathfrak{H}})| \leq c_u \|F\|_{L^2(\Omega)}
\] for any $F \in \mathbb{D}^{1,2}$.  If $u \in \dom \,\delta$, the random variable $\delta(u)$ is defined by the duality relationship 
\begin{equation} \label{dua}
\mathbb{E}(F\delta(u)) = \mathbb{E} (\langle DF, u \rangle_{\mathfrak{H}}) \, ,
\end{equation}
which is valid for all $F \in \mathbb{D}^{1,2}$.  
In a similar way, for each integer $k\ge 2$, we define the iterated divergence operator $\delta^k$ through the duality relationship 
\begin{equation} \label{dua2}
\mathbb{E}(F\delta^k(u)) = \mathbb{E}  \left(\langle D^kF, u \rangle_{\mathfrak{H}^{\otimes k}} \right),
\end{equation}
valid for any $F \in \mathbb{D}^{k,2}$, where $u\in  {\rm Dom}\, \delta^k \subset L^2(\Omega; \mathfrak{H}^{\otimes k})$.

Let $\gamma$  be the standard Gaussian measure on $\R$.  The Hermite polynomials $\{H_m(x), m\ge 0\}$  form a complete orthonormal system in $L^2(\R,\gamma)$ and  any function $g\in L^2(\R,\gamma)$ admits an  orthogonal expansion  of the form (\ref{hexp}). 
 If $g$ has Hermite rank $d$, for any integer $1\le k\le d$, we define the operator $T_k$ by
\begin{equation} \label{t1a}
T_k(g)(x) = \sum_{m=d}^\infty c_m H_{m-k}(x) \,.
\end{equation}
To simplify the notation we will write $T_k(g) =g_k$.

Suppose that $F$ is a random variable in the first Wiener chaos of $W$ of the form $F= I_1(\varphi)$, where $\varphi \in \HH$ has norm one.  Then  $g_k(F)$ has the representation
   \begin{equation}\label{g.intrep}
   	 g(F) = \delta^k( g_k(F) \varphi^{\otimes k}) \,.
   \end{equation}
   Moreover, if  $g(F) \in \DD^{j,p}$ for some $j\ge 0$ and $p>1$, then $g_k(F) \in \DD^{j+k, p} $. We refer to \cite{NZ} for the proof of these results. 
   
   Consider $\mathfrak{H}=\R$,  the probability space $(\Omega,  \mathcal{F}, P)= (\R, \mathcal{B}(\R), \gamma)$  and the isonormal Gaussian process $W(h)=h$.   For any $k\ge0$ and $p\ge 1$, denote by $\mathbb{D}^{k,p}(\R,\gamma)$ the  corresponding Sobolev spaces of functions.    Notice that if $F=I_1(\varphi)$ is an element in the first Wiener chaos  with $\|\varphi\|_{\HH} =1$, then $g\in \mathbb{D}^{k,p}(\R,\gamma)$ if and only if $g(F)\in \mathbb{D}^{k,p}$.

\subsection{Stein's method}
We refer to  \cite{ChenGoldShao} for a complete presentation of this topic.
Let  $h: \mathbb{R} \to \mathbb{R}$ be a Borel function such that $h \in L^1(\R, \gamma)$. The ordinary differential equation
      \begin{equation} \label{stein}
      f'(x) - xf(x) = h(x) - \mathbb{E}(h(Z))
      \end{equation}
is called the Stein's equation associated with $h$. The function 
\[
f_h(x):= e^{x^2/2}\int_{-\infty}^x (h(y) - \mathbb{E}(h(Z)))e^{-y^2/2} dy
\]
 is the unique solution to the Stein's equation satisfying $\lim_{|x| \to \infty} e^{-x^2/2} f_h(x) = 0$. Moreover, if $h$ is bounded by $1$, $f_h$ satisfies  $\|f_h \|_\infty \leq \sqrt{\pi /2}$ and 
    $   \|f_h' \|_\infty \leq 2$. 
    On the other hand,   if $h\in  {\rm Lip}(1)$ ($h$ is Lipschitz with a Lipschitz constant bounded by $1$), then 
    $f_h$ is continuously differentiable,  $   \|f_h' \|_\infty  \le \sqrt{2/\pi}$ and (see \cite[Lemma 3]{stein}) $ \|f_h''\| _\infty\le 2$.
We refer  to  \cite{np-book} and the references therein for a complete proof of these results. 

We recall that the total variation distance between the laws of two random variables $F,G$ is defined by
\[
d_{\rm TV}(F,G) = \sup_{B \in \mathcal{B}(\mathbb{R})}|P(F \in B) - P(G \in B)| \,,
\]
 where the supremum runs over all Borel sets $B \subset \mathbb{R}$.  Substituting $x$ by $F$  in  Stein's equation (\ref{stein}) and
 using the   estimate for  $ \|f_h' \|_\infty$    lead  to the fundamental estimate
\begin{equation}  \label{equ83}
d_{\rm TV}(F,Z) \le  \sup_{f\in  \mathcal{C}^1(\R),    \| f' \|_\infty \le 2 } | \EE  (f'(F)- Ff(F)) | \,.
\end{equation}
Furthermore, the Wasserstein distance between the laws of two random variables $F,G$ is defined by
\[
d_{\rm W}(F,G) = \sup_{ f \in {\rm Lip} (1)} | \E(f(F))- \E(f((G ))| \,
\]
and  using Stein's equation leads to
\begin{equation} \label{W}
d_{\rm W}(F,G) \le \sup_{f \in  \mathcal{F}_W}  | \EE  (f'(F)- Ff(F)) | \,,
\end{equation}
where $\mathcal{F}_W$ is the set of functions  $f \in \mathcal{C}^2(\R)$ such that $ \|f_h' \|_\infty  \le \sqrt{2/\pi}$ and $  \|f_h''\| _\infty\le 2$.

In the framework of an isonormal Gaussian process $W$, we  can use Stein's equation to estimate  the  total variation distance between a random variable $F = \delta(u)$ and $Z$.  A basic result is given in the next proposition (see \cite{eulalia,np-book}), which is an easy consequence of  (\ref{equ83}) and the duality relationship (\ref{dua}).
 
\begin{proposition}\label{prop2}
	Assume that $u\in {\rm Dom} \,\delta$,  $F=\delta(u) \in \mathbb{D}^{1,2}$ and $\EE(F^2)=1$.  Then,
	\begin{eqnarray*}
	d_{\rm TV} (F,Z) 	\le   2 \sqrt{\Var( \langle DF, u \rangle_{\mathfrak{H}} )}  \,.
	\end{eqnarray*}
\end{proposition}

An iterative  application of  the  Stein-Malliavin approach leads to the following result, which requires the random variable $F$ to be three times differentiable (see  \cite[Proposition 3.2.]{NZ}).

\begin{proposition}\label{prop2a}
Assume that $u\in {\rm Dom} \,\delta$,  $F=\delta(u) \in \mathbb{D}^{3,2}$ and $\EE(F^2)=1$.  Then,
\[
	d_{TV} (F ,Z)  \leq  (8+\sqrt{32\pi}) \Var( \langle DF, u \rangle_{\mathfrak{H}} )+\sqrt{2\pi}|\mathbb{E}(F^3)|+ \sqrt{32\pi}\mathbb{E}(|D_u F|^2)+ 4\pi \mathbb{E}(|D^3 _uF|),
\]
where we have used the notation $D_uF=\langle u, DF \rangle_{\mathfrak{H}}$ and  $ D_u^{i+1} F = \langle u, D(D_u^i F) \rangle_{\mathfrak{H}}  $ for  $i \ge 1$.
\end{proposition}

In the next proposition we present a new estimate for the Wasserstein's distance between a random variable $F=\delta^2(v)$ and a $N(0,1)$ random variable obtained using  Stein's method and Malliavin calculus.

\begin{proposition} \label{prop3}
Assume that $v \in \dom \, \delta^2$, $F= \delta^2(v) \in \mathbb{D}^{2,2}$ and $\mathbb{E}(F^2)=1$. Then, 
$$
d_{W}(F,Z) \leq  \sqrt{2/\pi}  \sqrt{\Var \left( \langle D^2F,v \rangle_{\HH^{\otimes 2}}\right)}+2 \mathbb{E}\left(\left|\langle DF\otimes DF,v \rangle_{\HH^{\otimes 2}}\right|\right).
 $$
\end{proposition}

\begin{proof}
By the duality relation (\ref{dua2}),  $\mathbb{E}\left(F\delta^2(v)\right)=\mathbb{E}\left(\langle D^2F,v \rangle_{\HH^{\otimes2}}\right)$.
As a consequence,  using (\ref{W}) we can write
 \begin{align*}
    d_{W}(F,Z) & \leq \sup_{f \in \mathcal{F}_W} |\mathbb{E}\left(f'(F)\right)-\mathbb{E}\left(Ff(F)\right)|=\sup_{f \in \mathcal{F}_W}|\mathbb{E}\left(f'(F)\right)-\mathbb{E}\left(\delta^2 (v) f(F)\right)| \\
   & =\sup_{f \in \mathcal{F}_W} |\mathbb{E}\left(f'(F)\right)-\mathbb{E}\left(\langle D^2(f(F)),v\rangle_{\HH^{\otimes 2}}\right)| \\
   & =\sup_{f\in \mathcal{F}_W}|\mathbb{E}\left(f'(F)\right)-\mathbb{E}\left( f'(F) \langle D^2F,v\rangle_{\HH^{\otimes2}}\right)-\mathbb{E}\left( f''(F) \langle DF\otimes DF,v \rangle _{\HH^{\otimes2}} \right)| \\
   & \leq  \sqrt{2/\pi}  \mathbb{E}\left(|1-\langle D^2F,v\rangle_{\HH^{\otimes 2}}|\right)+2\mathbb{E}\left(|\langle DF\otimes DF,v \rangle_{\HH^{\otimes 2}}
   |\right).
\end{align*}
Now, since $1=\mathbb{E}\left(F^2\right)=\mathbb{E}\left(F\delta^2(v)\right)=\mathbb{E}\left(\langle D^2F,v\rangle_{\HH^{\otimes2}}\right)$, using Cauchy-Schwarz inequality, we get 
\begin{align*}
    \mathbb{E}\left(|1-\langle D^2F,v\rangle_{\HH^{\otimes2}} |\right
    )\leq \sqrt{\mathbb{E} \left(\left|\mathbb{E}\left(\langle D^2F,v\rangle_{\HH^{\otimes2}}\right)-\langle D^2F,v\rangle_{\HH^{\otimes2}}\right| ^2\right)}=\sqrt{\Var(\langle D^2F,v\rangle_{\HH^{\otimes2}})} \ ,
\end{align*}
which concludes our proof.
\end{proof}

\subsection{Some basic inequalities}

In this subsection we recall  several inequalities proved in \cite{NZ} (see Lemmas 6.6, 6.7 and 6.8), which can be deduced from  the Brascamp-Lieb inequality (see   \cite{bl}) or just using
  H\"older's and Young's convolution inequalities. 

 \begin{lemma}
 Fix an integer $M\ge 2$. Let  $f$ be a non-negative function on the integers and set ${\bf k} = (k_1, \dots, k_M)$. Then, we have:
 \begin{itemize}
 \item[(i)]    For any vector ${\bf v} \in \R^M$  whose  components are $1$ or $-1$
\begin{equation}  \label{equ6a}
  \sum_{  {\bf k} \in \mathbb{Z}^M}  f({\bf k}  \cdot  {\bf v} )   \prod _{j=1} ^M f(k_j)    \le C \left(\sum_{k\in \mathbb{Z}} f(k)^{1+ \frac 1M}\right)^M.
\end{equation}
\item[(ii)] For any vector ${\bf v} \in \R^M$  whose  components are  $0$, $1$ or $-1$, assuming $\sum_{k\in \mathbb{Z}}  f(k)^2 <\infty$, 
  \begin{equation}  \label{equ6b}
  \sum_{  {\bf k} \in \mathbb{Z}^M}  f({\bf k}  \cdot  {\bf v} )   \prod _{j=1} ^M f(k_j)    \le C \left(\sum_{k\in \mathbb{Z}} f(k)\right)^{M-1}.
\end{equation} 
  \item[(iii)] Suppose $M\ge 3$. Let  ${\bf v}, {\bf w} \in \R^M$  be linearly independent vectors,  whose  components are  $0$, $1$ or $-1$. Suppose $\sum_{k\in \mathbb{Z}}  f(k)^2 <\infty$. Then,
  \begin{equation}  \label{equ6c}
  \sum_{  {\bf k} \in \mathbb{Z}^M}    f({\bf k}  \cdot  {\bf v} )  f({\bf k}  \cdot  {\bf w} )  \prod _{j=1} ^M f(k_j)    \le C \left(\sum_{k\in \mathbb{Z}} f(k)\right)^{M-2}.
  \end{equation}
  \end{itemize}
\end{lemma}

 \section{Some remarks on Wiener chaos expansions}
In this section we present some useful results on Wiener chaos expansions. 
We first recall a formula for the expectation of the product of  multiple stochastic integrals.
\begin{lemma}  \label{lem1}
Let $q_i \ge 1$ be integers,   and consider functions $f_i \in \HH^{\odot q_i}$, $i=1,\dots, M$. Then,
\[
 \E\left( \prod _{i=1}^M I_{q_i}( f_i)  \right)
= \sum_{  \beta \in  \mathcal{D}_q }   C_{q,\beta}      \left( \otimes_{i=1}^M f_i\right)_\beta,
\]
where
 \[
C_{q,\beta}  = \frac {  \prod_{i=1}^M q_i! }{ \prod_{1\le j<k\le M} \beta_{jk} !},
\]
  $ \mathcal{D}_q $ is the set of nonnegative integers $\beta_{jk}$,  $1\le j<k \le M$ satisfying
$$
q_i=  \sum_{j \, {\rm or} \, k =i} \beta_{jk}, \quad i=1\dots, M\ ,
$$
and 
$    \left( \otimes_{i=1}^M f_i\right)_\beta$ denotes the contraction of  $\beta_{jk}$ indexes between $f_j$ and $f_k$, for all $1\le j<k \le M$.
\end{lemma}

\begin{proof}
The product formula for multiple stochastic integrals (see, for instance,  \cite[Theorem 6.1.1]{PT}, or  formula (2.1) in \cite{BN}  for $M=2$)   says that
\begin{equation}  \label{product}
 \prod _{i=1}^M I_{q_i}( f_i)  
= \sum_{  \mathcal{P}, \psi }    I_{\gamma_1 + \cdots  +\gamma_M} \left(   \left( \otimes_{i=1}^M f_i\right)_{\mathcal{P}, \psi} \right),
\end{equation}  
where  $\mathcal{P}$ denotes the set of all partitions $\{1,\dots, q_i\} = J_i \cup \left(\cup_{ k=1,\dots, M, k \not=i} I_{ik}\right) $, where for any  $i,k =1,\dots,  M$, $I_{ik}$ and $I_{ki}$ have the same cardinality, $\psi_{ik}$ is a bijection between  $ I_{ik} $
 and $I_{ki}$ and $\gamma_i = | J_i|$. Moreover,  $ \left( \otimes_{i=1}^M f_i\right)_{\mathcal{P}, \psi}$ denotes the contraction of
 the indexes $\ell$ and $\psi_{ik} (\ell)$ for any  $\ell  \in I_{ik}$ and any  $i,k =1\dots, M$.
 Then, the expectation $\E\left( \prod _{i=1}^M I_{q_i}( f_i)   \right)$  corresponds to the case $\gamma_1= \cdots = \gamma_M =0$, and, if  we specify the number of partitions for fixed cardinalities $\beta_{jk}$,  we obtain the desired formula.
\end{proof}

 \subsection{Convergence of truncated expansions}
In general, given a random variable $F\in L^2(\Omega)$   with chaos expansion
(\ref{chaos}),  the fact that $\E( |F|^p) <\infty$ for some $p>2$ does not imply that the chaos expansion converges in $L^p(\Omega)$.
The next proposition  provides a partial result in this direction  for $p=2M$ and in the one-dimensional case, assuming that all the coefficients are nonnegative.

\begin{proposition}  \label{prop1}
 Consider a function $g\in L^2(\RR, \gamma)$, with an expansion of the form $g(x)=\sum_{q=0}^\infty  c_q H_q(x)$.
Suppose that $c_q \ge 0$ for each $q\ge 0$ and $g\in L^{2M}(\RR,\gamma)$ for some $M\ge 1$.  Consider the truncated sequence
\begin{equation} \label{truncated}
g^{(N)} := \sum_{q=0}^N c_q H_q.
\end{equation}
 Then $( g^{(N)})^M$  converges in $L^2(\RR,\gamma)$ to $g^{M}$.  
\end{proposition}

\begin{proof}
The proof will be done by induction on $M$.  The result is clearly true for $M=1$. Suppose that $M\ge 2$ and the result  holds for $M-1$.
Using the product formula for Hermite polynomials, which is a particular case of (\ref{product}),  we can write
\begin{align*}
(g^{(N)})^M &=    \sum_{q_1,\dots ,q_M=0}^N   \prod_{i=1}^M c_{q_i} H_{q_i}  \\
&=   \sum_{q_1,\dots ,q_M=0}^N  \left( \prod_{i=1}^M c_{q_i} \right)   \sum_{(\beta , \gamma)\in  \widehat{\mathcal{D}}_q} C_{q,\beta,\gamma}
H_{ \gamma_1 + \cdots + \gamma_M}\ ,
\end{align*}
where 
 \[
C_{q,\beta,\gamma}  = \frac {  \prod_{i=1}^M q_i! } { \prod_{i=1}^M  \gamma_i!  \prod_{1\le j<k\le M} \beta_{jk} !}\ ,
\]
and  $ \widehat{\mathcal{D}}_q $ is the set of nonnegative integers $\beta_{jk}$,  $1\le j<k \le M$ and $\gamma_i$, $1\le i\le M$, satisfying
\begin{equation} \label{beta}
q_i= \gamma_i+  \sum_{j \, {\rm or} \, k =i} \beta_{jk} ,\quad i=1,\dots,M\ .
\end{equation}
As a consequence, we obtain
$$
(g^{(N)})^M= \sum_{m=0} ^\infty d_{m,N} H_m\ ,
$$
where
$$
d_{m,N}= \sum_{q_1, \dots, q_M=0}^N \left(   \prod_{i=1}^M c_{q_i} \right) \sum_{(\beta,\gamma)\in   \widehat{\mathcal{D}}_q, \gamma_1 + \cdots + \gamma_M =m}C_{q,\beta,\gamma}\ .
$$
The function $g^M$ belongs to $L^2(\RR,\gamma)$. Therefore, it will have an expansion of the form
\[
g^M= \sum_{m=0} ^\infty d_{m} H_m\ .
\]
In order to compute the coefficients $d_m$, taking into account that $gH_m \in L^2(\RR,\gamma)$ and, by the induction hypothesis, $(g^{(N)})^{M-1}$ converges to $g^{M-1}$ in  $ L^2(\RR,\gamma)$ as $N\rightarrow \infty$, we can write
\begin{align*}
d_m =  \frac 1{m!} \E \left( g^M H_m \right) = \lim_{N\rightarrow \infty} \frac 1{m!} \E \left( g (g^{(N)})^{M-1}H_m\right).
\end{align*}
To compute the expectation  $\E \left( g (g^{(N)})^{M-1}H_m\right)$ we need the chaos expansion of  $(g^{(N)})^{M-1}H_m$:
\begin{align*}
(g^{(N)})^{M-1} H_m &=   \sum_{q_1,\dots ,q_{M-1}=0}^N  \prod_{i=1}^{M-1} c_{q_i}   \sum_{(\beta', \gamma')\in \widehat{\mathcal{D}}'_q} C_{q,\beta',\gamma'}
H_{ \gamma'_1 + \cdots + \gamma'_M}\ ,
\end{align*}
where 
 \[
C_{q,\beta' ,\gamma'}  = \frac {  m! \prod_{i=1}^{M-1} q_i! } {  \prod _{i=1}^M \gamma'_i !  \prod_{1\le j<k\le M} \beta'_{jk} !}\ ,
\]
 and $\widehat{\mathcal{D}}'_q$ is the set of $\beta$'s and $\gamma$'s such that   (\ref{beta}) holds  for $i=1, \dots, M-1$ and 
 $$
 m= \gamma_M + \sum_{j\, {\rm or } \, k =M} \beta'_{jk}\ .
 $$
As a consequence,
\begin{align*}
 \E \left( g (g^{(N)})^{M-1}H_m \right) = \sum_{q =0} ^\infty  q! c_{q}\sum_{q_1,\dots ,q_{M-1}=0}^N  \prod_{i=1}^{M-1} c_{q_i}   \sum_{(\beta' \gamma') \in \widehat{ \mathcal{D}}'_q, \gamma'_1 + \cdots+ \gamma'_M =q} C_{q,\beta', \gamma'}
 \end{align*}
 and, taking into account that the coefficients $c_q$ are nonnegative and  putting $q=q_M$,
   \begin{align*}
 d_m&=    \sum_{q_1,\dots ,q_{M}=0}^\infty  \prod_{i=1}^{M} c_{q_i}   \sum_{(\beta', \gamma')\in  \widehat{\mathcal{D}}'_q,
 \gamma'_1 + \cdots+ \gamma'_M =q_M
 } \frac {  \prod_{i=1}^{M} q_i! } {  \prod _{i=1}^M \gamma'_i !  \prod_{1\le j<k\le M} \beta'_{jk} !}\ .
\end{align*}
We claim that for any $(\beta',\gamma') \in \widehat{\mathcal{D}}'_q$ there exist a unique element   $(\beta,\gamma) \in \widehat{\mathcal{D}}_q$ such that
$$
\prod_{i=1}^M  \gamma_i!  \prod_{1\le j<k\le M} \beta_{jk} !=\prod_{i=1}^M  \gamma'_i!  \prod_{1\le j<k\le M} \beta'_{jk} !\ .
$$
 Indeed, it suffices to take $\beta_{jk} =\beta'_{jk}$ if $1\le j<k \le M-1$,  $\gamma_i = \beta'_{iM}$ for $i=1,\dots, M-1$,  $\gamma_M = \gamma'_M$, and $\beta_{jM} = \gamma'_j$ for
 $1\le j\le M-1$.
It follows that $\lim_{N\rightarrow \infty} d_{m,N} = d_m$.   This implies that
$(g^{(N)})^{M}$ converges in $L^2(\RR,\gamma)$ to $g^{M}$ and allows us to complete the proof.
\end{proof}

\subsection{The absolute value operator.}  Recall that $A$, defined in (\ref{ma2}) is the operator acting on $L^2(\R,\gamma)$ which replace the Hermite coefficients by its absolute values. Clearly, for any integer $k\ge 0$,  and for any $g\in \mathbb{D}^{k,2} (\R,\gamma)$, we have
\[
\| A(g) \| _{k,2} = \| g \| _{k,2}\ .
\]
Therefore, $g $ belongs to  $\mathbb{D}^{k,2} (\R,\gamma)$ if and only if $A(g) \in\mathbb{D}^{k,2} (\R,\gamma)$. If we consider functions in $L^p(\R,\gamma)$ for some real number $p>2$, we do not know whether $g\in L^p(\R,\gamma)$ implies $A(g)\in L^p(\R,\gamma)$. However, the following result holds.
\begin{lemma} Suppose that
 $A(g) \in \DD^{k,2M} (\R,\gamma)$ for some integers $M\ge 2$ and $k\ge 0$. Then $g \in \DD^{k,2M} (\R,\gamma)$. 
 \end{lemma}
 \begin{proof}
We will show the result only for $k=0$, the case $k\ge 1$  being similar.  Let $g=\sum_{q=d} ^\infty c_q H_q$ and define $g_+=\sum_{q=d} ^\infty c_q \mathbf{1}_{\{q: c_q>0\}} H_q$ and $g_-=\sum_{q=d} ^\infty c_q \mathbf{1}_{\{q: c_q<0\}} H_q$. Then $g=g_++g_-$.  We will show that $g_+\in L^{2M}(\R,\gamma)$, and in the same way one can prove that $g_-\in L^{2M}(\R,\gamma)$.
Using Proposition  \ref{prop1},  we can write
\begin{align*}
\E \left( g_+^{2M} \right) &=  \lim_{N\rightarrow \infty} \E \left( (g_+^{(N)})^{2M} \right)\\
&=
 \sum_{q_1,\dots ,q_{2M}=0}^\infty  \left( \prod_{i=1}^{2M} c_{q_i}\mathbf{1}_{\{q: c_q>0\}}  \right)    \sum_{\beta\in \mathcal{D}_q} 
  \frac {  \prod_{i=1}^{2M} q_i! } {      \prod_{1\le j<k\le 2M} \beta_{jk} !}\ ,
 \end{align*}
  where $ \mathcal{D}_q $ is the set of nonnegative integers $\beta_{jk}$,  $1\le j<k \le 2M$, satisfying
$q_i=   \sum_{j \, {\rm or} \, k =i} \beta_{jk},$ $i=1,\dots,2M$.
 Clearly, this implies that  $\E \left( g_+^{2M} \right) \le \E \left( A(g)^{2M} \right)<\infty$.
 \end{proof}
 
 \noindent

 The next lemma provides a criterion for a function $g$ to satisfy $A(g) \in  \mathbb{D}^{\ell,M}(\gamma)$ for  integers $\ell\ge 0$, $M\ge 3$.
 
 \begin{lemma} \label{lem3.3}  Fix  integers $\ell\ge 0$ and  $M\ge 3$. Let $g$ be a function in $g\in \mathbb{D}^{\ell,2} (\R,\gamma)$, with Hermite expansion $g= \sum_{k=0} ^\infty c_q H_q$. 
 Then,   $A(g) \in \mathbb{D}^{\ell,M}(\R, \gamma)$ if
   \begin{equation} \label{q3}
     \sum_{q=0}^{\infty} |c_q|q^{\frac\ell 2 -\frac 1 4} \sqrt{q!} (M-1)^{\frac  q2 }<\infty .
     \end{equation}
     \end{lemma}
     
  \begin{proof}
  We have
 \[
 D^\ell A^{(N)}  (g)= \sum_{q=\ell} ^N  |c_q| q(q-1) \cdots (q-\ell+1)H_{q-\ell}.
 \]
 Applying the estimate  (see, for instance, \cite{Lar})
 \[
 \|H_q \| _{L^M(\R, \gamma)} = c(M) q^{-\frac 14} \sqrt{ q!} (M-1)^{\frac q2}  (1+ O(q^{-1})),
 \]
 we obtain 
  \begin{align*}
\|  D^\ell A^{(N)}  (g)\|_{L^M(\R,\gamma)} &\le
c(M)  \Bigg ( | c_\ell|  
  \sum_{q=\ell} ^N  |c_q| q(q-1) \cdots (q-\ell+1)   (q-\ell)^{-\frac 14} \\
  & \qquad \times  \sqrt{ (q-\ell)!} (M-1)^{\frac {q-\ell}2}  (1+ O(q^{-1})) \Bigg) \\
  &\le c(M,\ell)  \Bigg ( | c_\ell|  +
  \sum_{q=\ell} ^N  |c_q| q ^{\frac \ell 2-\frac 14}   \sqrt{  q!} (M-1)^{\frac {q-\ell}2}  (1+ O(q^{-1})) \Bigg).
 \end{align*}
  Therefore, taking into account that  $A^{(N)}(g)$ converges in $L^2(\Omega)$ to $A(g)$ as $N$ tends to infinity, we conclude that 
  $\mathbb{E}(|D^\ell A(g)|^M) < \infty$ if  (\ref{q3}) holds.
     \end{proof}

 \section{Proof of Theorem \ref{thm1}} 
 
\begin{proof}
 Consider a centered stationary Gaussian  family of random variables   $X=\{ X_n , n\ge 0\}$  with unit variance and covariance
    $\rho(k) = \EE(X_0 X_k)$ for $k \ge 0$.  We put $\rho(-k) = \rho(k)$ for $k<0$.
    Suppose that $\HH$ is a Hilbert space and $e_i \in \HH$, $i \ge 0$,  are elements such that, for each $i,j \ge 0$, we have
    $\langle e_i, e_j \rangle_{\HH}= \rho(i-j)$. In this situation, if  $\{ W(\phi): \phi \in \HH\}$ is an isonormal Gaussian process, then
      the sequence  $X=\{ X_n , n\ge 0\}$  has the same law as $\{ W(e_n), n \ge 0\}$ and we can assume, without any loss of generality, that
      $X_n = W(e_n)$. 

      Consider the sequence   $	  F_n := \frac{1}{\sqrt{n}} \sum_{j=1}^n g(X_j) $ introduced in (\ref{yn}), where $g\in L^2(\R, \gamma)$ has  Hermite rank $d\ge 2$ and let 
	$\sigma_n^2= \EE (F_n^2)$.    Under condition  (\ref{h1}), it is well known that as $n\to \infty$, $\sigma_n^2 \to \sigma^2$,  where $\sigma^2$ has been defined in (\ref{bm.sig}). 
	Set $Y_n =\frac {F_n} {\sigma_n}$.  Notice that $\sigma>0$ implies that  $\sigma_n$ is bounded below for $n$ large enough. 
	Taking into account  (\ref{g.intrep}), we have  the representation $Y_n = \delta (\frac 1{ \sigma_n}u_n)$, where 
\begin{equation} \label{un}
u_n  =   \frac{1}{\sqrt{n}} \sum_{j=1}^n g_1(X_j) e_j,
\end{equation}
and $g_1$ is the shifted function introduced in  (\ref{t1a}).
		 	 
As a consequence of Proposition \ref{prop2}, we have the estimate
\begin{align}
	d_{TV} (Y_n ,Z) & \leq  2\sqrt{ {\rm Var} ( \langle DY_n,  \frac 1{\sigma_n}u_n \rangle_{\HH})} \nonumber\\
	& \le C \sqrt{ {\rm Var} ( \langle DF_n,  u_n \rangle_{\HH})} .  \label{yn2.est}
\end{align}
Then, we can write
 \begin{align*}
    \langle DF_n,u_n \rangle_{\HH} =\frac{1}{n} \sum_{i,j=1}^n g'(X_i)g_1(X_j)\rho(i-j).
    \end{align*}
  The random variable   $g'(X_i)  g_1(X_j)$ belongs to $L^2(\Omega)$, but we do  not know its chaos expansion. For this reason, we need to use a limit argument.   We have
  \[
     \langle DF_n,u_n \rangle_{\HH} =\lim_{N\rightarrow \infty} \Phi_{n,N},
     \]
     where the convergence holds in $L^1(\Omega)$ and
     \[
     \Phi_{n,N}= \frac{1}{n}\sum_{i,j=1}^n  \sum_{q_1,q_2 =d} ^N c_{q_1}c_{q_2}q_1  H_{q_1-1}(X_i)H_{q_2-1}(X_j)\rho(i-j).
     \]
Therefore, by Fatou's lemma
      \begin{align*}
   \Var \left( \langle DF_n,u_n \rangle_{\HH} \right) &  =\E \left(\langle DF_n,u_n \rangle_{\HH}^2 \right) - (\E \left(\langle DF_n,u_n \rangle_{\HH}\right))^2 \\
   &\le  \liminf_{N\rightarrow \infty} \left( \E ( \Phi_{n,N}^2) - ( \E(\Phi_{n,N}))^2 \right) \\
   &=  \liminf_{N\rightarrow \infty} \Var  (\Phi_{n,N}).
     \end{align*}
We can write
 \begin{align} \nonumber
 \Var  (\Phi_{n,N}) = &\frac{1}{n^2} \sum_{i_1,i_2,i_3,i_4=1}^n \sum_{q_1,q_2,q_3,q_4=d}^{N}  q_1 q_3c_{q_1}c_{q_2} c_{q_3}c_{q_4}  \rho(i_1-i_2)\rho(i_3-i_4)\\   \label{eq6}
 & \times \mathrm{Cov} (H_{q_1-1}(X_{i_1})H_{q_2-1}(X_{i_2}),H_{q_3-1}(X_{i_3})H_{q_4-1}(X_{i_4})) .
\end{align}
The next step is to compute the covariance appearing in the previous formula. To do this we will write the Hermite polynomials in terms of stochastic integrals and apply Lemma \ref{lem1}. That is,
 \begin{align*}
& \mathrm{Cov} (H_{q_1-1}(X_{i_1})H_{q_2-1}(X_{i_2}),H_{q_3-1}(X_{i_3})H_{q_4-1}(X_{i_4})) \\
&=\mathrm{Cov} (I_{q_1-1}(e_{i_1} ^{\otimes(q_1-1)} )I_{q_2-1}(e_{i_2} ^{\otimes(q_2-1)} ), I_{q_3-1}(e_{i_3} ^{\otimes(q_3-1)} )I_{q_4-1}(e_{i_4} ^{\otimes(q_4-1)} )) \\
&= \E\left(I_{q_1-1}(e_{i_1} ^{\otimes(q_1-1)} )I_{q_2-1}(e_{i_2} ^{\otimes(q_2-1)} ) I_{q_3-1}(e_{i_3} ^{\otimes(q_3-1)} )I_{q_4-1}(e_{i_4} ^{\otimes(q_4-1)} ) \right) \\
& \quad - \E\left(I_{q_1-1}(e_{i_1} ^{\otimes(q_1-1)} )I_{q_2-1}(e_{i_2} ^{\otimes(q_2-1)} )  \right) 
\E \left(I_{q_3-1}(e_{i_3} ^{\otimes(q_3-1)} )I_{q_4-1}(e_{i_4} ^{\otimes(q_4-1)} ) \right)
\end{align*}
and using  Lemma \ref{lem1}, 
 \begin{align}  \nonumber
&
 \E\left(I_{q_1-1}(e_{i_1} ^{\otimes(q_1-1)} )I_{q_2-1}(e_{i_2} ^{\otimes(q_2-1)} ) I_{q_3-1}(e_{i_3} ^{\otimes(q_3-1)} )I_{q_4-1}(e_{i_4} ^{\otimes(q_4-1)} ) \right)\\  \label{eq3}
&=\sum_{\beta \in \mathcal{D}_q}  C_{q,\beta}  \prod_{1\le j< k\le 4}   \rho(i_j -i_k)^{\beta_{jk}},
\end{align}
where
\[
C_{q,\beta}  = \frac {  \prod_{j=1}^4 (q_j-1)! }{ \prod_{1\le j< k\le 4}  \beta_{jk} !}
\]
and  $ \mathcal{D}_q $ is the set of nonnegative integers $\beta_{jk}$,  satisfying
\begin{equation}
q_\ell-1= \sum_{ j \, {\rm or} \, k =\ell} \beta_{jk}, \quad {\rm for}  \quad 1\le \ell \le 4.  \label{eq4} 
\end{equation}
On the other hand,
 \begin{align}  \nonumber
 &\E\left(I_{q_1-1}(e_{i_1} ^{\otimes(q_1-1)} )I_{q_2-1}(e_{i_2} ^{\otimes(q_2-1)} )  \right)
\E \left(I_{q_3-1}(e_{i_3} ^{\otimes(q_3-1)} )I_{q_4-1}(e_{i_4} ^{\otimes(q_4-1)} ) \right) \\  \label{eq2}
&= (q_1-1)! (q_3-1)!\rho^{q_1-1} (i_1-i_2)   \rho^{q_3-1} (i_3-i_4),
 \end{align}
 if $q_1=q_2$ and $q_3=q_4$, and zero otherwise.
 Notice that  (\ref{eq2}) is precisely the term in the sum  (\ref{eq3}) with $\beta_{12} =q_1-1$, $\beta_{34}=q_3-1$ and $\beta_{13}=\beta_{14}=\beta_{23}=\beta_{24}=0$. As a consequence, we obtain
 \begin{equation}      
\mathrm{Cov} (H_{q_1-1}(X_{i_1})H_{q_2-1}(X_{i_2}),H_{q_3-1}(X_{i_3})H_{q_4-1}(X_{i_4}))     \label{eq5}
= \sum_{\beta \in \mathcal{D}'_q}  C_{q,\beta}  \prod_{1\le j< k\le 4}   \rho(i_j -i_k)^{\beta_{jk}},
\end{equation}
where   $ \mathcal{D}'_q $ is the set of elements $(\beta_1, \dots, \beta_6)$, where the $\beta_k$'s are nonnegative integers satisfying (\ref{eq4})  and
\[
\beta_{13}+\beta_{14}+\beta_{23}+\beta_{24}\ge 1.
\]
Substituting (\ref{eq5}) into (\ref{eq6}) yields
 \begin{align*} 
 \Var  (\Phi_{n,N}) = &\frac{1}{n^2} \sum_{i_1,i_2,i_3,i_4=1}^n \sum_{q_1,q_2,q_3,q_4=d}^{N}  
  \sum_{\beta \in \mathcal{D}'_q} 
 C_{q,\beta}
 q_1 q_3c_{q_1}c_{q_2} c_{q_3}c_{q_4} \\   
 & \times  \rho^{\beta_{12}+1}(i_1- i_2)  \rho^{\beta_{13}}(i_1- i_3) \rho^{\beta_{14}}(i_1- i_4)
 \rho^{\beta_{23}}(i_2- i_3) \rho^{\beta_{24}}(i_2- i_4) \rho^{\beta_{34}+1}(i_3- i_4).
\end{align*}
Replacing $\beta_{12}+1$ and $\beta_{34}+1$ by $\beta_{12}$ and $\beta_{34}$, the above equality can be rewritten as
 \begin{align*} 
 \Var  (\Phi_{n,N}) = \frac{1}{n^2} \sum_{i_1,i_2,i_3,i_4=1}^n \sum_{q_1,q_2,q_3,q_4=d}^{N}  
  \sum_{\beta \in \mathcal{E}_q} 
K_{q,\beta}
 c_{q_1}c_{q_2} c_{q_3}c_{q_4}   \prod_{1\le j< k\le 4}   \rho(i_j -i_k)^{\beta_{jk}},
\end{align*}
where
\[
K_{q,\beta}  = \frac {    q_1! (q_2-1)! q_3! (q_4-1)! }{ (\beta_{12}-1)! \beta_{13}! \beta_{14}! \beta_{23}! \beta_{24}! (\beta_{34}-1)!}
\]
and  $ \mathcal{E}_q $ is the set of  nonnegative integers $\beta_{jk}$, $ 1\le j<k \le 4$,  satisfying
$\beta_{13}+\beta_{14}+\beta_{23}+\beta_{24}\ge 1$, $\beta_{12}\ge 1$, $\beta_{34}\ge 1$ and
$$
q_\ell= \sum_{ j \, {\rm or} \, k =\ell} \beta_{jk}, \quad {\rm for}  \quad 1\le \ell \le 4.   
$$
This leads to the estimate
\[
\Var  (\Phi_{n,N}) \le \sup_\beta A_{n,\beta}   \sum_{q_1,q_2,q_3,q_4=d}^{N} 
 \sum_{\beta \in \mathcal{E}_q} 
K_{q,\beta}
 |c_{q_1}c_{q_2} c_{q_3}c_{q_4}|,
 \]
 where
 \[
 A_{n,\beta}= \frac{1}{n^2} \sum_{i_1,i_2,i_3,i_4=1}^n 
 \prod_{1\le j< k\le 4}  | \rho(i_j -i_k)|^{\beta_{jk}},
 \]
 and the supremum is taken over all sets of  nonnegative integers $\beta_{jk}$, $ 1\le j<k \le 4$,   satisfying
$\beta_{13}+\beta_{14}+\beta_{23}+\beta_{24}\ge 1$, $\beta_{12}\ge 1$, $\beta_{34}\ge 1$, $\beta_{jk}\le d$ for $  1\le j<k \le 4$ and
$$
d \le  \sum_{ j \, {\rm or} \, k =\ell} \beta_{jk}, \quad {\rm for}  \quad 1\le \ell \le 4.   
$$
To complete the proof we need  to show the following claims:
\begin{itemize}
\item[(a)]   We have
\begin{equation} \label{ecu12}
\sum_{q_1,q_2,q_3,q_4=d}^{\infty} 
 \sum_{\beta \in \mathcal{E}_q} 
K_{q,\beta}
 |c_{q_1}c_{q_2} c_{q_3}c_{q_4}| <\infty.
 \end{equation}
 \item[(b)]  If $d=2$, then  $ \sup_\beta A_{n,\beta} $ is bounded by   a constant times the right-hand side of
(\ref{ecu1}).
 \item[(c)]  If $d\ge 3$, then  $ \sup_\beta A_{n,\beta} $ is bounded by  a constant times the right-hand side of
(\ref{ecu2}).
\end{itemize}

\medskip
\noindent
{\it Proof of (\ref{ecu12}):}  The main idea here is to identify the sum in (\ref{ecu12}) as the variance of a truncated function composed with a fixed random variable $X_1$.
From our previous computations it follows that
\begin{align*}
\sum_{q_1,q_2,q_3,q_4=d}^{N} 
 \sum_{\beta \in \mathcal{E}_q} 
K_{q,\beta}
 |c_{q_1}c_{q_2} c_{q_3}c_{q_4}| &=
 \sum_{q_1,q_2,q_3,q_4=d}^{N}   q_1 q_3  |c_{q_1}c_{q_2} c_{q_3}c_{q_4}|  \\
 &\times 
\mathrm{Cov} (H_{q_1-1}(X_1)H_{q_2-1}(X_1),H_{q_3-1}(X_1)H_{q_4-1}(X_1)) \\
&= \Var  (  A(g')^{(N)}(X_1) A(g_1)^{(N)} (X_1)),
\end{align*}
where for each integer $N\ge d$, we denote by $A(g')^{(N)}$   and $A(g_1)^{(N)}$  the truncated expansions  of $A(g')$  and $A(g_1)$, respectively, introduced in (\ref{truncated}).   By Proposition  \ref{prop1},  $(A(g')^{(N)})^2$   and  $(A(g_1)^{(N)})^2$ are convergent in $L^2(\RR, \gamma)$
to $A(g')^2$ and $A(g_1)^2$, respectively. Therefore,
\[
\sum_{q_1,q_2,q_3,q_4=d}^{\infty} 
 \sum_{\beta \in \mathcal{E}_q} 
K_{q,\beta}
 |c_{q_1}c_{q_2} c_{q_3}c_{q_4}|    =  \Var  ( A(g')(X_1) A(g_1) (X_1)) <\infty.
 \]
 
 \medskip
 \noindent
 {\it Proof of (b):} We will use ideas from graph theory to show the bound in the first part of Theorem 1.
 Recall the supremum is taken over all sets of  nonnegative integers $\beta_{jk}$, $ 1\le j<k \le 4$,   satisfying
$\beta_{13}+\beta_{14}+\beta_{23}+\beta_{24}\ge 1$, $\beta_{12}\ge 1$, $\beta_{34}\ge 1$, $\beta_{jk}\le 2$ for $  1\le j<k \le 4$ and
\begin{equation} \label{m5}
2 \le  \sum_{ j \, {\rm or} \, k =\ell} \beta_{jk}, \quad {\rm for}  \quad 1\le \ell \le 4. 
\end{equation}
The exponents $\beta_{jk}$ induce an unordered simple graph on the set of vertices $V= \{ 1,2,3,4\}$ by putting an edge between $j$ and $k$  if $\beta_{jk} \not=0$. There are edges connecting the  pairs of vertices $(1,2)$ and $(3,4)$ and condition $\beta_{13}+\beta_{14}+\beta_{23}+\beta_{24}\ge 1$ means that the graph is connected. Without any loss of generality, we can assume that there is an edge between the vertices $2$ and $3$. Then, condition 
(\ref{m5}) implies that  the degree of each vertex is at least two. The worse case is when the number of edges is minimal and the corresponding  nonzero coefficients  $\beta_{jk}$ are equal to one.  So far we have edges in $(1,2)$, $(3,4)$ and $(2,3)$. There must be more edges because each vertex must have at least degree two. There are two possible cases:

\begin{itemize}
\item[(i)] $\beta_{14} =1$. In this case we have
  \[
 A_{n,\beta}\leq \frac{1}{n^2} \sum_{i_1,i_2,i_3,i_4=1}^n 
 \left| \rho(i_1- i_2)  \rho(i_2- i_3) \rho(i_3- i_4) \rho(i_1- i_4) \right|.
 \]
 After making the change of variables  $i_1=i_1$, $k_1=i_1-i_2$, $k_2=i_2-i_3$ and $k_3=i_3-i_4$ and  using the inequality (\ref{equ6a}) with $M=3$ and $v=(1,1,1)$, we obtain 
\[
  A_{n,\beta}   
 \le \frac{1}{n} \sum_{ |k_i| \le n, i=1,2,3}
 \left| \rho(k_1)  \rho(k_2) \rho(k_3) \rho(k_1+k_2+k_3) \right|  \le \frac{C}{n} \left(\sum_{|k| \leq n} |\rho(k)|^{ \frac 43}\right)^3.
\]
\item[(ii)] Suppose that we add two more edges to  the graph formed by the edges  $(1,2)$, $(2,3)$ and $(3,4)$. In  this case, we obtain
  \[
 A_{n,\beta}\leq \frac{1}{n^2} \sum_{i_1,i_2,i_3,i_4=1}^n 
 \left| \rho(i_1- i_2)  \rho(i_2- i_3) \rho(i_3- i_4) \rho(i_{\alpha_1}- i_{\beta_1})\rho(i_{\alpha_2}- i_{\beta_2}) \right|.
 \]
 Making the change of variables $i_1=i_1$, $k_1=i_1-i_2$, $k_2=i_2-i_3$ and $k_3=i_3-i_4$, we obtain
   \[
 A_{n,\beta}\leq \frac{1}{n}   \sum_{ |k_i| \le n, i=1,2,3}
 \left| \rho(k_1)  \rho(k_2) \rho(k_3) \rho( {\bf k} \cdot  {\bf v})  \rho( {\bf k} \cdot  {\bf w}) \right|,
 \]
 where ${\bf v} $ and ${\bf w}$ are two linearly independent vectors in $\mathbb{Z}^3$ and ${\bf k} =(k_1,k_2,k_3)$.
  Using (\ref{equ6c}), we obtain 
  \[
 A_{n,\beta}\leq  \frac{C}{n} \sum_{|k|\leq n} 
 \left|   \rho(k) \right|,
 \]
  which completes the proof  of (b).
 \end{itemize}
  
   \medskip
 \noindent
 {\it Proof of (c):}   This estimate can be obtained by exactly the same arguments as in the proof of Theorem  4.5 in \cite{NZ}. We omit the details.
   \end{proof} 
    

      \begin{remark}
    We can show that both bounds in (\ref{ecu1}) are not comparable. In the particular case $|\rho(k)| \sim |k|^{-\alpha}$ as $|k| \rightarrow \infty$, with $\alpha >\frac 12$, we obtain:
    $$
      d_{\rm TV}(Y_n , Z)  \le
      \left\{
         \begin{array}{ll}  
             Cn^{1-2\alpha } & {\rm if} \quad \frac 12 <\alpha < \frac 23, \\
            C n^{-\frac \alpha 2}  & {\rm if} \quad  \frac 23 \le \alpha <1, \\
            C n^{-\frac 12} (\log n)^{\frac 12} & {\rm if} \quad  \alpha = 1,\\
                      C n^{-\frac 12}  & {\rm if} \quad  \alpha > 1.
         \end{array}
      \right.
    $$
    \end{remark}

\section{Proof of Theorem \ref{thm2}}

 \begin{proof}
 As in the proof of Theorem \ref{thm1}, we can assume that     $X_n = W(e_n)$, where  $e_i \in \HH$, $i \ge 0$   are elements in a Hilbert space $\HH$ such that, for each $i,j \ge 0$, we have
    $\langle e_i, e_j \rangle_{\HH}= \rho(i-j)$ and  $W=\{W(\phi): \phi \in \HH\}$ is an isonormal Gaussian process.

      Consider the sequence   $	  F_n := \frac{1}{\sqrt{n}} \sum_{j=1}^n g(X_j) $ introduced in (\ref{yn}), where $g\in L^2(\R, \gamma)$ has  Hermite rank $d= 2$ and let 
	$\sigma_n^2= \EE (F_n^2)$.    Set $Y_n =\frac {F_n} {\sigma_n}$. 
	Taking into account  (\ref{g.intrep}), we have  the representation $Y_n = \delta ^2(\frac 1{ \sigma_n}v_n)$, where 
\begin{equation} \label{vn}
v_n  =   \frac{1}{\sqrt{n}} \sum_{j=1}^n g_2(X_j) e_j \otimes e_j.
\end{equation}
		Under condition  (\ref{h1}), it is well known that as $n\to \infty$, $\sigma_n^2 \to \sigma^2$,  where $\sigma^2$ has been defined in (\ref{bm.sig}). 	 
As a consequence of Proposition \ref{prop3}, we have the estimate
\begin{align}
d_{W}(Y_n,Z) \leq  C  \sqrt{\Var \left( \langle D^2F_n,v_n \rangle_{\HH^{\otimes 2}}\right)}+C \mathbb{E}\left(|\langle DF_n\otimes DF_n,v_n \rangle_{\HH^{\otimes 2}}|\right).
\end{align}
Therefore, we need to estimate the quantities
  $  \Var(\langle D^2F_n,v _n\rangle_{\HH^{\otimes 2}})$
  and $  \mathbb{E}\left(|\langle DF_n\otimes DF_n,v_n \rangle_{\HH^{\otimes 2}}|\right)$. 
  
  \medskip
  \noindent
  {\it (i) Estimation of  $ \Var(\langle D^2F_n,v_n \rangle_{\HH^{\otimes 2}})$}.
We will follow  similar arguments as in the proof of Theorem \ref{thm2}. First, we write
 \begin{align*}
    \langle D^2F_n,v_n \rangle_{\HH^{\otimes 2}} =\frac{1}{n} \sum_{i,j=1}^n g''(X_i)g_2(X_j)\rho^2(i-j).
    \end{align*}
  Using a limit argument, we obtain
  \[
     \langle D^2F_n,V_n \rangle_{\HH^{\otimes 2}} =\lim_{N\rightarrow \infty} \Phi_{n,N},
     \]
     where the convergence holds in $L^1(\Omega)$ and
     \[
     \Phi_{n,N}= \frac{1}{n}\sum_{i,j=1}^n  \sum_{q_1,q_2 =2} ^N c_{q_1}c_{q_2}q_1(q_1-1)  H_{q_1-2}(X_i)H_{q_2-2}(X_j)\rho^2(i-j).
     \]
Therefore, by Fatou's lemma
$$
   \Var ( \langle D^2F_n,v_n \rangle_{\HH^{\otimes 2}} ) \le  \liminf_{N\rightarrow \infty} \Var  (\Phi_{n,N}).
$$
We can write
 \begin{align} \nonumber
 \Var  (\Phi_{n,N}) = &\frac{1}{n^2} \sum_{i_1,i_2,i_3,i_4=1}^n \sum_{q_1,q_2,q_3,q_4=2}^{N}  q_1 (q_1-1) q_3(q_3-1)c_{q_1}c_{q_2} c_{q_3}c_{q_4}  \rho^2(i_1-i_2)\rho^2(i_3-i_4)\\   \label{eq6'}
 & \times \mathrm{Cov} (H_{q_1-2}(X_{i_1})H_{q_2-2}(X_{i_2}),H_{q_3-2}(X_{i_3})H_{q_4-2}(X_{i_4})) .
\end{align}
 With a very similar calculation as in the proof of Theorem \ref{thm1}, we have

 \begin{equation}      
\mathrm{Cov} (H_{q_1-1}(X_{i_1})H_{q_2-1}(X_{i_2}),H_{q_3-1}(X_{i_3})H_{q_4-1}(X_{i_4}))     \label{eq5'}
= \sum_{\beta \in \mathcal{D}'_q}  C_{q,\beta}  \prod_{1\le j< k\le 4}   \rho(i_j -i_k)^{\beta_{jk}},
\end{equation}
where   $ \mathcal{D}'_q $ is the set of  nonnegative integers $\beta_{jk}$,  $1\le j<k \le 4$,
satisfying 
\begin{equation}
q_\ell-2= \sum_{ j \, {\rm or} \, k =\ell} \beta_{jk}, \quad {\rm for}  \quad 1\le \ell \le 4  \label{eq7} 
\end{equation}  and
\[
\beta_{13}+\beta_{14}+\beta_{23}+\beta_{24}\ge 1.
\]
Substituting (\ref{eq5'}) into (\ref{eq6'}) yields
 \begin{align*} 
 \Var  (\Phi_{n,N}) = &\frac{1}{n^2} \sum_{i_1,i_2,i_3,i_4=1}^n \sum_{q_1,q_2,q_3,q_4=2}^{N}  
  \sum_{\beta \in \mathcal{D}'_q} 
 C_{q,\beta}
 q_1(q_1-1) q_3(q_3-1)c_{q_1}c_{q_2} c_{q_3}c_{q_4} \\   
 & \times  \rho^{\beta_{12}+2}(i_1- i_2)  \rho^{\beta_{13}}(i_1- i_3) \rho^{\beta_{14}}(i_1- i_4)
 \rho^{\beta_{23}}(i_2- i_3) \rho^{\beta_{24}}(i_2- i_4) \rho^{\beta_{34}+2}(i_3- i_4).
\end{align*}
Replacing $\beta_{12}+2$ and $\beta_{34}+2$ by $\beta_{12}$ and $\beta_{34}$, the above equality can be rewritten as
 \begin{align*} 
 \Var  (\Phi_{n,N}) = \frac{1}{n^2} \sum_{i_1,i_2,i_3,i_4=1}^n \sum_{q_1,q_2,q_3,q_4=2}^{N}  
  \sum_{\beta \in \mathcal{E}_q} 
K_{q,\beta}
 c_{q_1}c_{q_2} c_{q_3}c_{q_4}   \prod_{1\le j< k\le 4}   \rho(i_j -i_k)^{\beta_{jk}},
\end{align*}
where
\[
K_{q,\beta}  = \frac {   q_1!  (q_2-2)!q_3!(q_4-2)!}{ (\beta_{12}-2)! \beta_{13}! \beta_{14}! \beta_{23}! \beta_{24}! (\beta_{34}-2)!}
\]
and  $ \mathcal{E}_q $ is the set of  nonnegative integers $\beta_{jk}$, $ 1\le j<k \le 4$,  satisfying
$\beta_{13}+\beta_{14}+\beta_{23}+\beta_{24}\ge 1$, $\beta_{12}\ge 2$, $\beta_{34}\ge 2$ and
$$
q_\ell= \sum_{ j \, {\rm or} \, k =\ell} \beta_{jk}, \quad {\rm for}  \quad 1\le \ell \le 4.  
$$
We can write
\[
\Var  (\Phi_{n,N}) \le \sup_\beta A_{n,\beta}   \sum_{q_1,q_2,q_3,q_4=2}^{N} 
 \sum_{\beta \in \mathcal{E}_q} 
K_{q,\beta}
 |c_{q_1}c_{q_2} c_{q_3}c_{q_4}|,
 \]
 where
 \[
 A_{n,\beta}= \frac{1}{n^2} \sum_{i_1,i_2,i_3,i_4=1}^n 
 \prod_{1\le j< k\le 4}  | \rho(i_j -i_k)|^{\beta_{jk}},
 \]
 and the supremum is taken over all sets of  nonnegative integers $\beta_{jk}$, $ 1\le j<k \le 4$,   satisfying
$\beta_{13}+\beta_{14}+\beta_{23}+\beta_{24}\ge 1$, $\beta_{12}\ge 2$, $\beta_{34}\ge 2$, for $  1\le j<k \le 4$ and
$$
2 \le  \sum_{ j \, {\rm or} \, k =\ell} \beta_{jk}, \quad {\rm for}  \quad 1\le \ell \le 4. 
$$
 Then, in this case we have
  \[
 A_{n,\beta}\leq \frac{1}{n^2} \sum_{i_1,i_2,i_3,i_4=1}^n 
 \left| \rho(i_1- i_2)^2  \rho(i_{\alpha_1}- i_{\alpha_2})  \rho(i_3- i_4)^2 \right|
 \]
 where $\alpha_1 \in \{1,2\}$ and $\alpha_2 \in \{3,4\}$. After making the change $i_1=i_1$, $k_1=i_1-i_2$, $k_2=i_{\alpha_1}-i_{\alpha_2}$ and $k_3=i_3-i_4$, we obtain 
  \[
A_{n,\beta}\leq   \frac{1}{n} \sum_{|k_i |\le n, i=1,2,3}  
 \left| \rho(k_1)^2  \rho(k_2) \rho(k_3)^2 \right| \leq \frac{C}{n} \sum_{|k|\leq n} 
 \left|   \rho(k) \right|.
 \]
 Now, it is left to show that 
 \begin{align}\label{finite}
     \sum_{q_1,q_2,q_3,q_4=2}^{N} 
 \sum_{\beta \in \mathcal{E}_q} 
K_{q,\beta}
 |c_{q_1}c_{q_2} c_{q_3}c_{q_4}| < \infty.
 \end{align}
We have
  \begin{align*}
     & \sum_{q_1,q_2,q_3,q_4=2}^{N} 
 \sum_{\beta \in \mathcal{E}_q} 
K_{q,\beta}
 |c_{q_1}c_{q_2} c_{q_3}c_{q_4}| =\sum_{q_1,q_2,q_3,q_4=2}^{N} q_1 (q_1-1) q_3(q_3-1)|c_{q_1}c_{q_2} c_{q_3}c_{q_4} | \\ &\times \mathbb{E}\left( H_{q_1-2}(X_1)H_{q_2-2}(X_1)H_{q_3-2}(X_1)H_{q_4-2}(X_1)\right) \\  &= \mathbb{E}\left( (A(g'')^{(N)})^2 (A(g_2)^{(N)})^2\right).
 \end{align*}
 \quad
 By H\"{o}lder's inequality, we obtain
 \begin{align*}
     \sum_{q_1,q_2,q_3,q_4=2}^{N} 
 \sum_{\beta \in \mathcal{E}_q} 
K_{q,\beta}
 |c_{q_1}c_{q_2} c_{q_3}c_{q_4}| \leq \| A(g'')^{(N)}\|_{L^4(\mathbb{R},\gamma)}^{1/2}\| A(g_2)^{(N)}\|_{L^4(\mathbb{R},\gamma)}^{1/2}.
\end{align*}
From the hypothesis and the Proposition \ref{prop1}, $(A(g'')^{(N)})^2$ and $(A(g_2)^{(N)})^2$ converge to $A(g'')^2$ and $A(g_2)^2$ in $L^2(\mathbb{R},\gamma)$ respectively. Hence, (\ref{finite}) holds.

    \medskip
    
  \noindent
  {\it (ii) Estimation of $  \mathbb{E}\left(|\langle DF_n\otimes DF_n,v_n \rangle_{\HH^{\otimes 2}}|\right)$. } 
 We can write
  \[
 \langle DF_n\otimes DF_n,v_n \rangle_{\HH^{\otimes 2}} 
  = n^{-\frac 32}  \sum_{i,j,k=1}^n g'(X_i) g'(X_j) g_2(X_k)\rho(i-k) \rho(j-k).
 \]
 We have, in the $L^1(\Omega)$ sense,
  \[
     \langle DF_n,u_n \rangle_{\HH} =\lim_{N\rightarrow \infty} \Psi_{n,N},
     \]
  where
     \[
     \Psi_{n,N}=  n^{-\frac 32} \sum_{i,j,k=1}^n  \sum_{q_1,q_2,q_3 =2} ^N c_{q_1}c_{q_2} c_{q_3}q_1 q_2   H_{q_1-1}(X_i)H_{q_2-1}(X_j) H_{q_3-2}(X_k)\rho(i-k)\rho(j-k).
     \]
Therefore, by Fatou's lemma
\[
 \E\left( \langle DF\otimes DF,v \rangle_{\HH^{\otimes 2}} ^2 \right)   \le  \liminf_{N\rightarrow \infty} \E \left(  \Psi_{n,N}^2 \right).
\]
We can write
 \begin{align}  \nonumber
 \E \left(\Psi_{n,N}^2 \right) &= n^{-3} \sum_{i_1, \dots, i_6=1}^n  \sum_{q_1, \dots, q_6 =2}^N
 \left(\prod_{i=1}^6 c_{q_i}\right) q_1q_2 q_4q_5  \\ \nonumber
 & \qquad \times  \E\left( H_{q_1-1}(X_{i_1})H_{q_2-1}(X_{i_2}) H_{q_3-2}(X_{i_3}) H_{q_4-1}(X_{i_4})H_{q_5-1}(X_{i_5}) H_{q_6-2}(X_{i_6})  \right) \\
 & \qquad  \times \rho(i_1-i_3) \rho(i_2-i_3)   \rho(i_4-i_6) \rho(i_5-i_6). \label{bar2}
 \end{align}
 Using Lemma  \ref{lem1}, we obtain
 \begin{align} \nonumber
 & \E\left( H_{q_1-1}(X_{i_1})H_{q_2-1}(X_{i_2}) H_{q_3-2}(X_{i_3}) H_{q_4-1}(X_{i_4})H_{q_5-1}(X_{i_5}) H_{q_6-2}(X_{i_6})  \right) \\
 &= \sum_{\beta \in \mathcal{D}_q}  C_{q,\beta}  \prod_{ 1 \le j<k \le 6 }   \rho(i_j -i_k)^{\beta_{jk}}, \label{bar1}
\end{align}
where
\[
C_{q,\beta}  = \frac { (q_3-2)! (q_6-2)!  \prod_{j=1,2,4,5}^4 (q_j-1)! }{ \prod_{1 \le  j<k \le 6 }  \beta_{jk} !}
\]
and  $ \mathcal{D}_q $ is the set of  nonnegative integers $\beta_{jk}$, $ 1\le j<k \le 6$,  satisfying
\begin{align} \nonumber
q_\ell-1&=  \sum_{j \, {\rm or}\, k =\ell}  \beta_{jk}  , \quad {\rm for}\quad \ell=1,2,4,5\ , \\ \nonumber 
q_3-2&=  \sum_{j \, {\rm or}\, k =3}  \beta_{jk}  , \\   
q_6-2&=  \sum_{j \, {\rm or}\, k =6}  \beta_{jk}.  
\end{align}
Replacing (\ref{bar1}) into (\ref{bar2}) yields
 \begin{align*}  
 \E(\Psi_{n,N}^2) &= n^{-3} \sum_{i_1, \dots, i_6=1}^n  \sum_{q_1, \dots, q_6 =2}^N \sum_{\beta \in \mathcal{D}_q}  C_{q,\beta} 
  \left(\prod_{i=1}^6 c_{q_i}\right) q_1q_2 q_4q_5\\
 &  \qquad \times  \rho(i_1-i_3) \rho(i_2-i_3)   \rho(i_4-i_6) \rho(i_5-i_6)
  \prod_{j,k=1, j<k}^6   \rho(i_j -i_k)^{\beta_{jk}}. 
 \end{align*}
Substituting $\beta_{13}+1$,  $\beta_{23}+1$, $\beta_{46}+1$ and $\beta_{56}+1$ by
$\beta_{13}$,  $\beta_{23}$, $\beta_{46}$ and $\beta_{56}$, respectively, we can write
 \begin{align*}  
 \E(\Psi_{n,N}^2) &= n^{-3} \sum_{i_1, \dots, i_6=1}^n  \sum_{q_1, \dots, q_6 =2}^N \sum_{\beta \in \mathcal{E}_q}  K_{q,\beta} 
  \left(\prod_{i=1}^6 c_{q_i}\right) q_1q_2 q_4q_5
  \prod_{j,k=1, j<k}^6   \rho(i_j -i_k)^{\beta_{jk}},
 \end{align*}
 where
 \[
K_{q,\beta}  = \frac {   {\beta_{13} \beta_{23}\beta_{46}\beta_{56} (q_3-2)! (q_6-2)!  \prod_{j=1,2,4,5}^4 (q_j-1)! } }{ \prod_{j,k=1, j<k}^6  \beta_{jk} !}
\]
and  $ \mathcal{E}_q $ is the set of  nonnegative integers $\beta_{jk}$, $1 \le j<k \le 6$,  satisfying
\begin{equation}
q_\ell =  \sum_{j \, {\rm or}\, k =\ell}  \beta_{jk}  , \quad {\rm for}\quad \ell=1,\dots, 6.  \nonumber  
\end{equation}
Hence
$$
\E(\Psi_{n,N}^2) \le \sup_\beta A_{n,\beta}    \sum_{q_1, \dots, q_6 =2}^N \sum_{\beta \in \mathcal{E}_q}  K_{q,\beta} 
  \left(\prod_{i=1}^6 |c_{q_i}|\right) q_1q_2 q_4q_5,
$$
where
$$
A_{n,\beta}= n^{-3} \sum_{i_1, \dots, i_6=1}^n  \prod_{1\le j<k \le 6}   |\rho(i_j -i_k)|^{\beta_{jk}}
$$
and the supremum is taken over all sets of nonnegative integers  $\beta_{jk}$, $j,k=1,\dots, 6$, $j<k$,  satisfying
$\beta_{13}\ge 1$,  $\beta_{23}\ge 1$, $\beta_{46}\ge 1$,  $\beta_{56}\ge 1$ and 
\begin{equation} \label{bar4}
2 \le  \sum_{j \, {\rm or}\, k =\ell}  \beta_{jk}  , \quad {\rm for}\quad \ell=1,\dots, 6 \\ .   
\end{equation}
 As in the proof of Theorem \ref{thm1}, we can show that
 \begin{equation} \label{bar3}
  \sum_{q_1, \dots, q_6 =2}^\infty \sum_{\beta \in \mathcal{E}_q}  K_{q,\beta} 
  \left(\prod_{i=1}^6 |c_{q_i}|\right) q_1q_2 q_4q_5 <\infty.
  \end{equation}
  In fact,
  \begin{align*}
  &  \sum_{q_1, \dots, q_6 =2}^N \sum_{\beta \in \mathcal{E}_q}  K_{q,\beta} 
  \left(\prod_{i=1}^6 |c_{q_i}|\right) q_1q_2 q_4q_5 =
      \sum_{q_1, \dots, q_6 =2}^N  \left(\prod_{i=1}^6 |c_{q_i}|\right) q_1q_2 q_4q_5  \\
      &\times
      \E\left[ H_{q_1-1}(X_1)H_{q_2-1}(X_1) H_{q_3-2}(X_1) H_{q_4-1}(X_1)H_{q_5-1}(X_1) H_{q_6-2}(X_1) \right] \\
      &= \E [ (A(g')^{(N)})^4(X_1) (A(g_2)^{(N)})^2(X_1)],
      \end{align*}  
      where, as before, $A(g')^{(N)}$ and  $A(g_2)^{(N)}$ are the truncated expansions of $A(g')$ and $A(g_2)$, respectively.  By H\"older's inequality, we can write
     $$
     \sum_{q_1, \dots, q_6 =2}^N \sum_{\beta \in \mathcal{E}_q}  K_{q,\beta} 
  \left(\prod_{i=1}^6 |c_{q_i}|\right) q_1q_2 q_4q_5
  \le   \| A(g')^{(N)}\|_{L^6(\R,\gamma)}^{\frac 23}\| A(g_2)^{(N)}\|_{L^6(\R,\gamma)}^{\frac 13}.
  $$
          From our hypothesis and in view of Proposition  \ref{prop1}, $(A(g')^{(N)})^3$ and $ (A(g_2)^{(N)})^3$ converge in
      $L^2(\R,\gamma) $ to $A(g')$ and $A(g_2)$,  respectively.  Thus,  (\ref{bar3}) holds true.
 
 To complete the proof, it remains to show that,  
 $$
\sup_{\beta} A_{n,\beta} \le  C n^{-1} \left(\sum_{|k| \leq n} |\rho(k)| ^{\frac 32}\right)^4.
 $$
As in the proof of Theorem \ref{thm1}, in order to show this estimate we will make use of some ideas from graph theory. The exponents $\beta_{jk}$ induce an unordered simple graph on the set of vertices  $V=\{ 1,2,3,4,5,6\}$ by putting an edge between $j$ and $k$ whenever $\beta_{jk }\not=0$. Because  $\beta_{13}\ge 1$,  $\beta_{23}\ge 1$, $\beta_{46}\ge 1$ and  $\beta_{56}\ge 1$,  there  are edges connecting the pairs of  vertices $(1,3)$, $(2,3)$, $(4,6)$ and $(5,6)$. Condition (\ref{bar4}) means that the degree of each vertex is at least $2$.
   Then we consider two cases, depending whether graph is connected or not.
   
   \medskip
   \noindent
   {\it Case 1:} Suppose that  the graph is not connected. This implies that $\beta_{12} \ge 1$, $\beta_{45} \ge 1$ and there is no edge between the sets
   $V_1=\{ 1,2,3\}$ and $V_2=\{ 4,5,6\}$. The worse case is when $\beta_{12} = \beta_{13} =\beta_{23} =\beta_{45} = \beta_{46} =\beta_{56} =1$ and all the other exponents are zero. In this case we have the estimate
   \[
   A_{n,\beta} \le  n^{-1}  \left(    \sum_{ |k_1|, |k_2| \le n}
   | \rho(k_1) \rho(k_2) \rho(k_1-k_2)| \right)^2.
   \]
   Using (\ref{equ6a}), we obtain
    \[
   A_{n,\beta} \le  C n^{-1} \left(\sum_{|k| \leq n} |\rho(k)| ^{\frac 32}\right)^4.
   \]
 
    \medskip
   \noindent
   {\it Case 2:}  Suppose that  the graph is  connected. This means that there is an edge   connecting the sets $V_1$ and $V_2$. Suppose that
   $\beta_{\alpha_0\delta_0} \ge 1$, where $\alpha_0 \in \{1,2,3\}$ and $\delta_0\in \{4,5,6\}$.   We have then $5$ nonzero coefficients $\beta$:   $ \beta_{13}$, $\beta_{23}$, $\beta_{46} $, $\beta_{56}$ and $\beta_{\alpha_0\delta_0}$.
   Because all the edges have at least degree $2$, there must be  at least two more nonzero coefficients $\beta$. Let us denote them by
   $\beta_{\alpha_1\delta_1}$ and    $\beta_{\alpha_2\delta_2}$.
   
    Then, the worse case will be when $ \beta_{13} =\beta_{23} = \beta_{46} =\beta_{56} =\beta_{\alpha_0\delta_0}=\beta_{\alpha_1\delta_1}=\beta_{\alpha_2\delta_2}=1$ and all the other coefficients are zero. Consider the change of variables $i_1-i_3=k_1$, $i_2-i_3 =k_2$, $i_4-i_6=k_3$, $i_5-i_6= k_4$, $i_{\alpha_0}-i_{\delta_0} =k_5$. Then, 
    $i_{\alpha_1} - i_{\delta_1} = {\bf k} \cdot {\bf v}$ and
       $i_{\alpha_2} - i_{\delta_2} = {\bf k} \cdot {\bf w}$, 
    where ${\bf k} = (k_1,\dots, k_5)$ and ${\bf v}$, ${\bf w}$ are  $5$-dimensional linearly independent vectors whose components are  $0$, $1$ or $-1$. 
     Then, we can write, using (\ref{equ6c}) and H\"older's inequality,
        \begin{align*}
A_{n,\beta}  &\le  n^{-2}      \sum_{ |k_i| \le n, 2\le i\le 5}  
   \prod_{i=2}^5  | \rho(k_i) |  |\rho( {\bf k} \cdot {\bf v}) \rho( {\bf k} \cdot {\bf w})|  \le  Cn^{-2}   \left( \sum_{|k| \leq n} |\rho(k)| \right)^{3}   \\
  &\le   Cn^{-1}  \left(\sum_{|k| \leq n} |\rho(k)| ^{ \frac 32}\right)^4.
   \end{align*}
 \end{proof}
   
  \begin{remark}
    In the case $g(x)= x^2-1$, the  term $ \Var(\langle D^2F_n,v_n \rangle_{\HH^{\otimes 2}})$ is zero because  $ \langle D^2F_n,v_n \rangle_{\HH^{\otimes 2}}$ is deterministic,  and for the second term we get the estimate (\ref{optimal}).
  \end{remark}
 
    \begin{remark}
    We can show that both bounds in (\ref{equa1}) are not comparable. In the particular case $|\rho(k)| \sim |k|^{-\alpha}$ as $|k| \rightarrow \infty$, with $\alpha >\frac 12$, we obtain:
    $$
      d_{\rm W}(Y_n , Z)  \le
      \left\{
         \begin{array}{ll} 
             Cn^{\frac 32 -3\alpha} & {\rm if} \quad \frac 12 <\alpha \le \frac 35, \\
            C n^{-\frac \alpha 2}  & {\rm if} \quad \frac 35 <\alpha \le 1, \\
            C n^{-\frac 12} (\log n)^{\frac 12} & {\rm if} \quad  \alpha = 1,\\
                      C n^{-\frac 12}  & {\rm if} \quad  \alpha > 1.
         \end{array}
      \right.
    $$
    \end{remark}

 \section{Proof of Theorem \ref{thm1a}} 
 
\begin{proof} With the notation used in the proof of Theorem \ref{thm1} and using Proposition \ref{prop2a}, we can write
\begin{align}
	d_{TV} (Y_n ,Z) & \leq  (8+\sqrt{32\pi}) {\rm Var} (\langle D Y_n,  u_n/\sigma_n  \rangle_{\HH}) +\sqrt{2\pi}|\mathbb{E}(Y_n^3)|+ \sqrt{32\pi}\mathbb{E}(|D^2_{u_n/\sigma_n}Y_n|^2)\\
	& \qquad + 4\pi \mathbb{E}(|D^3_{u_n/\sigma_n}Y_n|) \nonumber \\
	&\leq C\left( {\rm Var} (\langle D F_n,  u_n  \rangle_{\HH}+|\mathbb{E}(F_n^3)|+  \mathbb{E}(|D^2_{u_n}F_n|^2)+  \sqrt{ \mathbb{E}(|D^3_{u_n}F_n|^2) } \right).\nonumber
\end{align}

Now, we want to estimate each of these terms separately.

\medskip
\noindent
{\it Step 1.}  \quad From Theorem \ref{thm1} we know that
\begin{equation}
 {\rm Var} (\langle D F_n,  u_n  \rangle_{\HH}  \leq  C n^{-1} \sum_{|k| \leq n} |\rho(k)|  + C n^{-1} \left(\sum_{|k| \leq n} |\rho(k)|^{\frac{4}{3}}\right)^ 3.
\end{equation}

\medskip\noindent
{\it Step 2. } \quad We claim that
\begin{equation}  \label{a1}
|\mathbb{E}(F_n^3)| \le  \frac{C}{\sqrt{n}}  \left(\sum_{|k| \leq n} |\rho(k)|^{ \frac 32}\right)^2.
\end{equation}
We can write
 \begin{align*}
    F_n^3=\frac{1}{n^{3/2}} \sum_{i,j,k=1}^n g(X_i)g(X_j)g(X_k).
    \end{align*}
    Truncating the Wiener chaos expansion of  the random variables  $g(X_i)$, as in the proof of  Theorem \ref{thm1}, we obtain
     \[
     F_n^3 =\lim_{N\rightarrow \infty} \Psi^3_{n,N} : =\lim_{N\rightarrow \infty}
 \frac{1}{\sqrt{n}}\sum_{i=1}^n  \sum_{q=2} ^N c_{q}  H_{q}(X_i),
     \] 
where the convergence holds in $L^2(\Omega)$ due to Proposition \ref{prop1} because    $g \in L^6(\R, \gamma)$.     
Therefore, 
      \begin{align*}
   \mathbb{E} ( F_n^3 ) 
   =  \lim_{N\rightarrow \infty} \mathbb{E}  (\Psi_{n,N}^3).
     \end{align*}
We can write
 \begin{align} \nonumber
 \mathbb{E} (\Psi^3_{n,N}) = &\frac{1}{n^{3/2}} \sum_{i_1,i_2,i_3=1}^n \sum_{q_1,q_2,q_3=2}^{N}  c_{q_1}c_{q_2} c_{q_3} \mathbb{E} (H_{q_1}(X_{i_1})H_{q_2}(X_{i_2})H_{q_3}(X_{i_3})) \\
 &= \frac{1}{n^{3/2}} \sum_{i_1,i_2,i_3=1}^n \sum_{q_1,q_2,q_3=2}^{N}  c_{q_1}c_{q_2} c_{q_3} \mathbb{E} \left( I_{q_1}(e_{i_1} ^{\otimes q_1} )I_{q_2}(e_{i_2} ^{\otimes q_2} ) I_{q_3}(e_{i_3} ^{\otimes q_3}) \right).
\end{align}
Using  Lemma \ref{lem1}, we obtain
 \begin{equation}
 \E\left(I_{q_1}(e_{i_1} ^{\otimes q_1} )I_{q_2}(e_{i_2} ^{\otimes q_2} ) I_{q_3}(e_{i_3} ^{\otimes q_3} )  \right)  \label{eq9}
=\sum_{\beta \in \mathcal{D}_q}  C_{q,\beta}  \prod_{1\le j< k\le 3}   \rho(i_j -i_k)^{\beta_{jk}},
\end{equation}
where
\[
C_{q,\beta}  = \frac {  \prod_{j=1}^3 q_j! }{ \prod_{1\le j< k\le 3}  \beta_{jk} !} 
\]
and  $ \mathcal{D}_q $ is the set of nonnegative integers $\beta_{jk}$,  $ 1\le j < k \le 3$,  satisfying
\begin{equation}
q_\ell= \sum_{ j \, {\rm or} \, k =\ell} \beta_{jk}, \quad {\rm for}  \quad 1\le \ell \le 3.   
\end{equation}
Then,
\[
|\mathbb{E}  (\Psi^3_{n,N})| \le \sup_\beta A_{n,\beta}   \sum_{q_1,q_2,q_3=2}^{N} 
 \sum_{\beta \in \mathcal{E}_q} 
C_{q,\beta}
 |c_{q_1}c_{q_2} c_{q_3}|,
 \]
 where
 \[
 A_{n,\beta}= \frac{1}{n^{3/2}} \sum_{i_1,i_2,i_3=1}^n 
 \prod_{1\le j< k\le 3}  | \rho(i_j -i_k)|^{\beta_{jk}},
 \]
 and the supremum is taken over all sets of  nonnegative integers $\beta_{jk}$, $ 1\le j<k \le 3$,   satisfying $\beta_{jk}\le 2$ for $  1\le j<k \le 3$ and
$$
2 \le  \sum_{ j \, {\rm or} \, k =\ell} \beta_{jk}, \quad {\rm for}  \quad 1\le \ell \le 3.
$$
It is easy to see that to satisfy the above conditions, $\beta_{jk} \geq 1$ for all $1\leq j < k\leq 3$. Hence, we have 
\[ 
 A_{n,\beta}\leq \frac{1}{n^{3/2}} \sum_{i_1,i_2,i_3=1}^n   | \rho(i_1 -i_2)\rho(i_1 -i_3)\rho(i_2 -i_3)| .\]
 After making the change  of variables $i_1=i_1$, $k_1=i_1-i_2$, $k_2=i_1-i_3$ and  using  the inequality (\ref{equ6a}) with $M=2$ and $v=(-1,1)$, we obtain 
\[ 
 A_{n,\beta}\leq \frac{1}{n^{1/2}} \sum_{|k_1|, |k_2|\leq n}  | \rho(k_1)\rho(k_2)\rho(k_2 -k_1)| \leq \frac{C}{\sqrt{n}}  \left(\sum_{|k| \leq n} |\rho(k)|^{ \frac 32}\right)^2.
\]
To complete the proof  of (\ref{a1}), we need  to show that:
\[
\sum_{q_1,q_2,q_3=2}^{\infty} 
 \sum_{\beta \in \mathcal{D}_q} 
C_{q,\beta}
 |c_{q_1}c_{q_2} c_{q_3}| <\infty.\]
In fact,
\[ 
\lim_{N\rightarrow \infty} \sum_{q_1,q_2,q_3=2}^{N} 
 \sum_{\beta \in \mathcal{D}_q} 
C_{q,\beta}
 |c_{q_1}c_{q_2} c_{q_3}| = \lim_{N\rightarrow \infty}   \mathbb{E}\left(A(g)^{N})^3\right)=
  \mathbb{E}\left((A(g))^3\right)
<\infty,
\]
taking into account Proposition   \ref{prop1} and the fact that  $A(g) \in L^6(\mathbb{R},\gamma)$.

\medskip \noindent
{\it Step 3.}  \quad We proceed now with the  estimation of $\mathbb{E}(|D_{u_n}^2F_n|^2)$. We can write
\[
    D_{u_n}F_n =\langle DF_n,u_n \rangle_{\HH} = \frac{1}{n}\sum_{i,j=1}^n g'(X_i)g_1(X_j)\rho(i-j)
    \]
    and
\[
    D(\langle DF_n,u_n \rangle_{\HH})  = \frac{1}{n} \sum_{i,j=1}^n (g''(X_i)g_1(X_j)e_i+g'(X_i)g_1'(X_j)e_j)\rho(i-j).
    \]
    Therefore,
    \begin{align}  \notag
     D_{u_n}^2F_n& =  \langle u_n,D( \langle DF_n,u_n \rangle_{\HH})\rangle_{\HH} 
     \\ & = \frac{1}{n^{3/2}} \sum_{i,j,k=1}^n (g''(X_i)g_1(X_j)g_1(X_k)\rho(i-k)+g'(X_i)g_1'(X_j)g_1(X_k)\rho(j-k))\rho(i-j).  \label{a2}
     \end{align}
    Because  the random variables  $g''(X_i)$, $g_1(X_j)$,   $g_1(X_k)$, $g'(X_i)$ and $g_1'(X_j)$ appearing in the above expression belong to  $L^2(\Omega)$, their truncated Wiener chaos expansions convergence in $L^2(\Omega)$, and, as a consequence,
     $D^2_{u_n}F_n =\displaystyle\lim_{N \to \infty} \Phi_{n,N} $ in probability, where
    \begin{align*}
    \Phi_{n,N} & =\frac{1}{n^{3/2}} \sum_{i_1,i_2,i_3=1}^n \sum_{q_1,q_2,q_3=2}^N c_{q_1}c_{q_2}c_{q_3}q_1(q_1-1)H_{q_1-2}(X_{i_1})H_{q_2-1}(X_{i_2}) H_{q_3-1}(X_{i_3}) \\ 
    &  \qquad \times \rho(i_1-i_2)\rho(i_1-i_3) \\ 
    & \qquad  +
     c_{q_1}c_{q_2}c_{q_3}q_1(q_2-1)H_{q_1-1}(X_{i_1})H_{q_2-2}(X_{i_2}) H_{q_3-1}(X_{i_3}) \rho(i_1-i_2)\rho(i_2-i_3).
\end{align*}
Making the change of variables $(q_1,q_2) \to (q_2,q_1)$ and $(i_1,i_2)\to(i_2,i_1)$ in the second sum allows us to put the two terms together, and we obtain
   \begin{align*}
    \Phi_{n,N} & =\frac{1}{n^{3/2}} \sum_{i_1,i_2,i_3=1}^n \sum_{q_1,q_2,q_3=2}^N c_{q_1}c_{q_2}c_{q_3}(q_1+q_2)(q_1-1)H_{q_1-2}(X_{i_1})H_{q_2-1}(X_{i_2}) H_{q_3-1}(X_{i_3}) \\ 
    &  \qquad \times \rho(i_1-i_2)\rho(i_1-i_3).
\end{align*}
 Therefore, by Fatou's lemma, 
 \begin{align*}
 \mathbb{E}\left(| D_{u_n}^2F_n|^2\right) \leq \liminf_{N \to \infty} \mathbb{E}\left(|\Phi_{n,N}^2|\right).   
\end{align*}
Then,
\begin{align*}
   | \Phi_{n,N}|^2 & = \frac{1}{n^{3}} \sum_{i_1,\dots,i_6=1}^n \sum_{q_1,\dots,q_6=2}^{N} C_q H_{q_1-2}(X_{i_1})H_{q_2-1}(X_{i_2}) H_{q_3-1}(X_{i_3})  \\ & \times H_{q_4-2}(X_{i_4}) H_{q_5-1}(X_{i_5})H_{q_6-1}(X_{i_6})\rho(i_1-i_2)\rho(i_1-i_3)\rho(i_4-i_5)\rho(i_4-i_6),
\end{align*}
where 
\[
C_q =c_{q_1} c_{q_2}c_{q_3}c_{q_4}c_{q_5}c_{q_6} (q_1+q_2)(q_1-1) (q_4+q_5)(q_4-1).
\]
Using the product formula for multiple integrals (see Lemma \ref{lem1}), we get 
\begin{align*}
\mathbb{E}\left(| \Phi_{n,N}|^2\right) & = \frac{1}{n^{3}} \sum_{i_1, \dots, i_6=1}^n \sum_{q_1,\dots,q_6=2}^{N}    \sum_{\beta \in \mathcal{D}_q} K_{q,\beta}\left(\prod_{1 \leq k< l \leq 6} \rho(i_k-i_l)^{\beta_{kl}}\right)\\
& \qquad \times   \rho(i_1-i_2)\rho(i_1-i_3)\rho(i_4-i_5)\rho(i_4-i_6),
\end{align*}
where 
\[
K_{q,\beta}=  \frac{(q_1 + q_2) (q_4+q_5) \prod_{j=1}^6c_{q_j} (q_j-1)! }{\prod_{1\leq k <l\leq 6}  \beta_{kl}!}
\]
and 
\[
    \mathcal{D}_q=\{(\beta_{kl})_{1\leq k <l \leq 6}: \sum_{k \textit{ or } l =j} \beta_{kl}= q_j-1 \text{ for } j= 2,3,5,6 \text{ and } \sum_{k \textit{ or } l =j} \beta_{kl}= q_j-2 \text{ for } j= 1,4\}.
 \]
Replacing $\beta_{jk}+1$ by  $\beta_{jk}$  for $(j,k) \in \{(1,2), (1,3), (4,5), (4,6)\}$, yields
\[
\mathbb{E}\left(| \psi_{n,N}|^2\right)  = \frac{1}{n^{3}} \sum_{i_1, \dots, i_6=1}^n \sum_{q_1, \dots, q_6=2}^{N}    \sum_{\beta \in \mathcal{C}_q } L_{q,\beta} \left(\prod_{1 \leq k< l \leq 6} \rho(i_k-i_l)^{\beta_{kl}}\right) ,
\]
where 
\[
L_{q,\beta}= \frac{(q_1+q_2)(q_4 +q_5)\prod_{i=1}^6c_{q_i}(q_i-1)!}{(\beta_{12}+1)!(\beta_{13}+1)!\beta_{14}!\beta_{15}!\beta_{16}!\beta_{23}!\beta_{24}!\beta_{25}!\beta_{26}!\beta_{34}!\beta_{35}!\beta_{36}!(\beta_{45}+1)!(\beta_{46}+1)!\beta_{56}!}
\]
and 
\[
    \mathcal{C}_q=\{(\beta_{kl})_{1\leq k <l \leq 6}: \sum_{k \textit{ or } l =j} \beta_{kl}= q_j \text{ for } j= 1,\dots,6 \text{ and }  \beta_{12}, \, \beta_{13}, \, \beta_{45}. \, \beta_{46} \geq 1\}.  
    \]
Then, we can write
\[
\mathbb{E}  \left(|\psi_{n,N})|^2\right) \le   \sup_{\beta\in \mathcal{C}_q}  A_{n,\beta}   \sum_{q_1,\dots,q_6=2}^{N}
 \sum_{\beta \in \mathcal{C}_q} 
|L_{q,\beta}|,
 \]
where 
\[
A_{n,\beta}  = \frac{1}{n^3} \sum_{i_1,i_2,i_3,i_4=1}^n \prod_{1 \leq j < k \leq 6}|\rho(i_i-i_k)|^{\beta_{jk}}
\]
and the supremum is taken over all sets of  nonnegative integers $\beta_{jk}$, $ 1\le j<k \le 6$,   satisfying
$ \beta_{12}, \, \beta_{13}, \, \beta_{45}, \, \beta_{46} \geq 1$,
 $\beta_{jk}\le 2$ for $  1\le j<k \le 6$ and
$$
2 \le  \sum_{ j \, {\rm or} \, k =\ell} \beta_{jk}, \quad {\rm for}  \quad 1\le \ell \le 6.
$$
Then, the estimation follows as in the proof of the last part of Theorem  \ref{thm2}.

Now, we need to show that 
\begin{align} \label{finite1} \sum_{q_1,\dots, q_6=2}^{\infty}  
 \sum_{\beta \in \mathcal{C}_q} 
|L_{q,\beta}| < \infty.
\end{align}
In fact,
\begin{align*}
\sum_{q_1, \dots,q_6=2}^{N}
 \sum_{\beta \in \mathcal{C}_q} 
|L_{q,\beta}|  =\sum_{q_1, \dots,q_6=2}^{N} \left(\prod_{i=1}^6 |c_{q_i}| \right) (q_1+q_2)(q_1-1)(q_3+q_4)(q_4-1) \\ \times \mathbb{E}\left(H_{q_1-2}(X_1)H_{q_2-1}(X_1)H_{q_3-1}(X_1)H_{q_4-2}(X_1)H_{q_5-1}(X_1)H_{q_6-1}(X_1)\right) \\
 = \mathbb{E}\left(A(g'')^{(N)})^2(A(g_1)^{(N)})^4\right) \leq \|A(g'')^{(N)}\|_{L^6(\mathbb{R},\gamma)}^{\frac{1}{3}}\|A(g_1)^{(N)}\|_{L^6(\mathbb{R},\gamma)}^{\frac{2}{3}}.
\end{align*}
Since $A(g) \in \mathbb{D}^{3,6}$, $(A(g'')^{(N)})^3$ and $(A(g_1)^{(N)})^3$ converge to $A(g'')$ and $A(g_1)$, respectively, in $L^2(\mathbb{R},\gamma)$ by (\ref{prop1}). Then, (\ref{finite1}) is true.

\medskip
\noindent
{\it Step 4.} \quad We proceed to the estimation of $\sqrt{ \mathbb{E}(|D_{u_n}^3F_n|^2) }$.
Taking the derivative in (\ref{a2}), yields
\begin{align*}
   D( D_{u_n}^2F_n) & = \frac{1}{n^{3/2}} \sum_{i,j,k=1}^n g'''(X_{i})g_1(X_{j})g_1(X_{k})\rho(i-j)\rho(i-k)e_i \\ & +g''(X_i)g_1'(X_j)g_1(X_k)\rho(i-j)\rho(i-k)e_j   +   g''(X_i)g_1(X_j)g_1'(X_k)\rho(i-j)\rho(i-k)e_k  \\ & +g''(X_i)g_1'(X_j)g_1(X_k)\rho(i-j)\rho(j-k)e_i   +g'(X_i)g_1''(X_j)g_1(X_k)\rho(i-j)\rho(j-k)e_j \\ & +g'(X_i)g_1'(X_j)g_1'(X_k)\rho(i-j)\rho(j-k)e_k.
\end{align*}
 This implies
\begin{align*}
  \langle u_n, D( D_{u_n}^2F_n \rangle_{\HH} &= \frac{1}{n^{2}} \sum_{i_1,i_2,i_3,i_4=1}^n  g'''(X_{i_1})g_1(X_{i_2})g_1(X_{i_3})g_1(X_{i_4})\rho(i_1-i_2)\rho(i_1-i_3)\rho(i_1-i_4) \\
  & \quad  +g''(X_{i_1})g_1'(X_{i_2})g_1(X_{i_3})g_1(X_{i_4})\rho(i_1-i_2)\rho(i_1-i_3)\rho(i_2-i_4) \\ 
  & \quad   +  
  g''(X_{i_1})g_1(X_{i_2})g_1'(X_{i_3})g_1(X_{i_4})\rho(i_1-i_2)\rho(i_1-i_3)\rho(i_3-i_4)  \\ 
  &  \quad  +g''(X_{i_1})g_1'(X_{i_2})g_1(X_{i_3})g_1(X_{i_4})\rho(i_1-i_2)\rho(i_2-i_3)\rho(i_1-i_4)   \\ 
  & \quad  +g'(X_{i_1})g_1''(X_{i_2})g_1(X_{i_3})g_1(X_{i_4})\rho(i_1-i_2)\rho(i_2-i_3)\rho(i_2-i_4) \\ 
  & \quad +g'(X_{i_1})g_1'(X_{i_2})g_1'(X_{i_3})g_1(X_{i_4})\rho(i_1-i_2)\rho(i_2-i_3)\rho(i_3-i_4).
\end{align*}
Notice that the second, third and fourth terms are identical. This allows us to write
\begin{align*}
  D^3_{u_n}F_n  &= \frac{1}{n^{2}} \sum_{i_1,i_2,i_3,i_4=1}^n  g'''(X_{i_1})g_1(X_{i_2})g_1(X_{i_3})g_1(X_{i_4})\rho(i_1-i_2)\rho(i_1-i_3)\rho(i_1-i_4) \\
  & \quad  +3g''(X_{i_1})g_1'(X_{i_2})g_1(X_{i_3})g_1(X_{i_4})\rho(i_1-i_2)\rho(i_1-i_3)\rho(i_2-i_4) \\ 
  & \quad  +g'(X_{i_1})g_1''(X_{i_2})g_1(X_{i_3})g_1(X_{i_4})\rho(i_1-i_2)\rho(i_2-i_3)\rho(i_2-i_4) \\ 
  & \quad +g'(X_{i_1})g_1'(X_{i_2})g_1'(X_{i_3})g_1(X_{i_4})\rho(i_1-i_2)\rho(i_2-i_3)\rho(i_3-i_4).
\end{align*}
Then, we have
  \[
      D^3_{u_n}F_n =\lim_{N\rightarrow \infty} \Phi_{n,N},
     \]
     where the convergence holds in probability and
     \begin{align*}
    \Phi_{n,N} & = \frac{1}{n^{2}}   \sum_{i_1,i_2,i_3,i_4=1}^n \sum_{q_1,q_2,q_3,q_4=2}^{N} C_q^{(1)}  H_{q_1-3}(X_{i_1})H_{q_2-1}(X_{i_2})
    H_{q_3-1}(X_{i_3}) H_{q_4-1}   (X_{i_4}) \\ &  \times     \rho(i_1-i_2)\rho(i_1-i_3)\rho(i_1-i_4)  
    \\ + & C_q^{(2)} H_{q_1-2}(X_{i_1})H_{q_2-2}(X_{i_2})    H_{q_3-1}(X_{i_3}) H_{q_4-1}(X_{i_4})  \rho(i_1-i_2)\rho(i_1-i_3)\rho(i_2-i_4)
    \\ + & C_q^{(3)} H_{q_1-1}(X_{i_1})H_{q_2-3}(X_{i_2})    H_{q_3-1}(X_{i_3}) H_{q_4-1}(X_{i_4})  \rho(i_1-i_2)\rho(i_2-i_3)\rho(i_2-i_4)
    \\ + & C_q^{(4)} H_{q_1-1}(X_{i_1})H_{q_2-2}(X_{i_2})   H_{q_3-2}(X_{i_3}) H_{q_4-1}(X_{i_4})  \rho(i_1-i_2)\rho(i_2-i_3)\rho(i_1-i_4)
\end{align*}
with
\begin{align*}
C_q^{(1)} &=c_{q_1}c_{q_2}c_{q_3}c_{q_4} q_1(q_1-1)(q_1-2) , \\
C_q^{(2)}&=3c_{q_1} c_{q_2} c_{q_3}c_{q_4}  q_1(q_1-1) (q_2-1) , \\
C_q^{(3)}&=c_{q_1}c_{q_2} c_{q_3}c_{q_4} q_1(q_2-1)(q_2-2) ,\\
C_q^{(4)}&=c_{q_1}c_{q_2} c_{q_3} c_{q_4}q_1(q_2-1) (q_3-1) .
\end{align*}
We can combine the first and third terms with the change of variables $(q_1,q_2) \to (q_2,q_1)$ and $(i_1,i_2)\to(i_2,i_1)$. In this way we obtain
 \begin{align*}
    \Phi_{n,N} & = \frac{1}{n^{2}}   \sum_{i_1,i_2,i_3,i_4=1}^n \sum_{q_1,q_2,q_3,q_4=2}^{N} \widetilde{C}_q^{(1)}  H_{q_1-3}(X_{i_1})H_{q_2-1}(X_{i_2}) 
    H_{q_3-1}(X_{i_3}) H_{q_4-1}(X_{i_4}) \\  &  \quad  \times      \rho(i_1-i_2)\rho(i_1-i_3)\rho(i_1-i_4)  
    \\   & \quad + \widetilde{C}_q^{(2)} H_{q_1-2}(X_{i_1})H_{q_2-2}(X_{i_2})    H_{q_3-1}(X_{i_3}) H_{q_4-1}(X_{i_4})  \rho(i_1-i_2)\rho(i_1-i_3)\rho(i_2-i_4)
     \\  &  \quad  +\widetilde{C}_q^{(3)} H_{q_1-1}(X_{i_1})H_{q_2-2}(X_{i_2})   H_{q_3-2}(X_{i_3}) H_{q_4-1}(X_{i_4})  \rho(i_1-i_2)\rho(i_2-i_3)\rho(i_1-i_4)\\
    & =: \Phi^{(1)}_{n,N} +\Phi^{(2)}_{n.N}+\Phi^{(3)}_{n.N} 
\end{align*}
with
\begin{align*}
\widetilde{C}_q^{(1)} &=c_{q_1}c_{q_2}c_{q_3}c_{q_4} (q_1+q_2) (q_1-1)(q_1-2) , \\
\widetilde{C}_q^{(2)}&=c_{q_1} c_{q_2} c_{q_3}c_{q_4}  3q_1(q_1-1)  (q_2-1), \\
\widetilde{C}_q^{(3)}&=c_{q_1} c_{q_2} c_{q_3}c_{q_4}  q_1(q_2-1) (q_3-1).
\end{align*}
Then,  by Fatou's lemma, 
\begin{align*}
    \mathbb{E}\left( |D^3_{u_n}F_n| ^2\right) \leq \liminf_{N \to \infty } \mathbb{E}\left( |\Phi_{n, N}|^2\right).
\end{align*}

We are going to treat each term $\Phi_{n,N}^{(i)}$, $i=1,2,3$, separately.

\medskip
\noindent
{\it Case $i=1$}. Let us first estimate   $\mathbb{E}\left( |\Phi^{(1)}_{n, N}|^2\right)$.
We have
\begin{align*}
\E\left( (\Phi^{(1)} _{n,N})^2  \right) &= \frac 1{n^4} \sum_{i_1, \dots, i_8=1}^n \sum_{q_1, \dots, q_8 =2}^N
M^{(1)}_q \E\left(  H_{q_1-3}(X_{i_1})H_{q_2-1}(X_{i_2}) 
    H_{q_3-1}(X_{i_3}) H_{q_4-1}(X_{i_4}) \right. \\
    & \left. \quad\times  H_{q_5-3}(X_{i_5})H_{q_6-1}(X_{i_6}) 
    H_{q_7-1}(X_{i_7}) H_{q_8-1}(X_{i_8}) \right) \\
    & \quad  \times     \rho(i_1-i_2)\rho(i_1-i_3)\rho(i_1-i_4)   \rho(i_5-i_6)\rho(i_5-i_7)\rho(i_5-i_8),
\end{align*}
where
\[
M^{(1)}_q=  \left( \prod_{j=1}^8 c_{q_j}\right)  (q_1+q_2) (q_1-1)(q_1-2)(q_5+q_6) (q_5-1)(q_5-2).
\]
This yields
\begin{align*}
\E \left(\Phi^{(1)} _{n,N})^2  \right) & \leq \frac{1}{n^{4}} \sum_{i_1, \dots ,i_8=1}^n \sum_{q_1, \dots q_8=2}^{N}     \sum_{\beta \in \mathcal{D}_q^{(1)}} K_{q,\beta}^{(1)} \left(\prod_{1 \leq k< l \leq 8} |\rho(i_k-i_l)|^{\beta_{kl}}\right)  \\
&\quad \times  |\rho(i_1-i_2)\rho(i_1-i_3)\rho(i_1-i_4)\rho(i_5-i_6)\rho(i_5-i_7)\rho(i_5-i_8) |,
\end{align*}
where 
\[
K_{q,\beta}^{(1)}= \frac{(q_1+q_2) ( q_5+ q_6)  \prod _{j=1}^8  | c_{q_j} |(q_j-1)!}{\prod_{1\leq k <l\leq 8} \beta_{kl}!},
\]
and 
\begin{align*}
 \mathcal{D}_q^{(1)}& =\{(\beta_{kl})_{1\leq k <l \leq 8}: \sum_{k \textit{ or } l =j} \beta_{kl}= q_j-1 \text{ for } j= 2,3,4,6,7,8 \\
 &\qquad \text{ and } \sum_{k \textit{ or } l =j} \beta_{kl}= q_j-3 \text{ for } j= 1,5\}.  
\end{align*}
Changing the exponents $\beta_{jk}+1$ in to $\beta_{jk}$ for $(j,k) \in \{ (1,2), (1,3), (1,4), (5,6), (5,7), (5,8)\}$, we can write
\[
\E\left( (\Phi^{(1)} _{n,N})^2 \right)   \leq \frac{1}{n^{4}} \sum_{i_1, \dots ,i_8=1}^n \sum_{q_1, \dots q_8=2}^{N}     \sum_{\beta \in \mathcal{C}_q^{(1)}} L_{q,\beta}^{(1)} \left(\prod_{1 \leq k< l \leq 8} |\rho(i_k-i_l)|^{\beta_{kl}}\right),
\]
where 
\[
L_{q,\beta}^{(1)}= \frac{(q_1+q_) ( q_5+ q_6)  \prod _{j=1}^8  | c_{q_j} |(q_j-1)!}{(\beta_{12}-1)!
(\beta_{13}-1)!(\beta_{14}-1)!(\beta_{56}-1)!(\beta_{57}-1)(\beta_{58}-1)!\prod_{(k,l) \in \mathcal{E}} \beta_{kl}!},
\]
with $\mathcal{E} = \{ (k,l): 1\le k<l \le 8,  (k,l) \not= (1,2), (1,3), (1,4), (5,6), (5,7), (5,8)\}$
and
\[
    \mathcal{C}_q^{(1)}=\{(\beta_{kl})_{1\leq k <l \leq 8}: \sum_{k \textit{ or } l =j} \beta_{kl}= q_j \text{ for } j= 1,\dots,8 \text{ and }  \beta_{12}, \, \beta_{13}, \, \beta_{14}, \beta_{56}, \beta_{57}, \beta_{58} \geq 1\}.
 \]
Then, we obtain
\[
\E \left( (\Phi^{(1)} _{n,N})^2 \right)  \le   \sup_{\beta\in \mathcal{C}^{(11)}_q}  A^{(1)}_{n,\beta}   \sum_{q_1,\dots,q_8=2}^{N}  
 \sum_{\beta \in \mathcal{C}^{(1)}_q} 
|L_{q,\beta}^{(1)}|,
 \]
where 
\[
A^{(1)}_{n,\beta}  = \frac{1}{n^4} \sum_{i_1, \dots, i_8=1}^n \prod_{1 \leq j \leq k \leq 8}|\rho(i_i-i_k)|^{\beta_{jk}}
\]
and the supremum is taken over all sets of  nonnegative integers $\beta_{jk}$, $ 1\le j<k \le 8$,   satisfying
 $\beta_{12}, \, \beta_{13}, \, \beta_{14}, \beta_{56}, \beta_{57}, \beta_{58} \geq 1$,
 $\beta_{jk}\le 2$ for $  1\le j<k \le 8$ and
$$
2 \le  \sum_{ j \, {\rm or} \, k =\ell} \beta_{jk}, \quad {\rm for}  \quad 1\le \ell \le 8.
$$
We need to estimate  $A^{(1)}_{n,\beta}$ and to show that 
\begin{equation}  \label{gr2}
 \sum_{q_1,\dots,q_8=2}^{\infty}  
 \sum_{\beta \in \mathcal{C}^{(1)}_q} 
L_{q,\beta}^{(1)} <\infty.
\end{equation}

\medskip
\noindent
{\it Estimation of $A^{(1)}_{n,\beta}$:} \quad 
We claim that  
 \begin{equation} \label{gr1}
\sup_{\beta} A^{(1)}_{n,\beta} \le  C n^{-1} \left(\sum_{|k| \leq n} |\rho(k)| ^{\frac 32}\right)^4.
 \end{equation}
As in the proof of Theorem \ref{thm2}, we will make use of ideas from graph theory. The exponents $\beta_{jk}$ induce an unordered simple graph on the set of vertices  $V=\{ 1,2,3,4,5,8\}$ by putting an edge between $j$ and $k$ whenever $\beta_{jk }\not=0$. Because  $\beta_{12}, \beta_{13}\ge 1$,  $\beta_{14}\ge 1$, $\beta_{56}\ge 1, \beta_{57}\ge 1$ and  $\beta_{58}\ge 1$, there  are edges connecting the pairs of vertices $(1,2)$, $(1,3)$, $(1,4)$, $(5,6)$,
$(5,7)$ and $(5,8)$. Condition (\ref{bar4}) means that the degree of each vertex is at least $2$.
   Then we consider two cases, depending whether graph is connected or not.
   
   \medskip
   \noindent
   {\it Case 1:} Suppose that  the graph is not connected.  This means that
    $\beta_{jk} =0$ if $j \in \{1,2,3,4\}$ and $k\in \{5,6,7,8\}$ and there is no edge between the sets
   $V_1=\{ 1,2,3,4\}$ and $V_2=\{ 5,6,7,8\}$. Therefore,
    \[
    A^{(1)}_{n,\beta} \le     (A^{(0)}_{n,\beta} )^2,
    \]
    where
     \[
A^{(0)}_{n,\beta}  = \frac{1}{n^2} \sum_{i_1, \dots, i_4=1}^n \prod_{1 \leq j \leq k \leq 4}|\rho(i_i-i_k)|^{\beta_{jk}}
\]
and the   nonnegative integers $\beta_{jk}$, $ 1\le j<k \le 4$,   satisfy
 $\beta_{12}, \, \beta_{13}, \, \beta_{14} \geq 1$,
 $\beta_{jk}\le 2$ for $  1\le j<k \le 4$ and
$$
2 \le  \sum_{ j \, {\rm or} \, k =\ell} \beta_{jk}, \quad {\rm for}  \quad 1\le \ell \le 4.
$$
     As a consequence,
    $\beta_{23}+ \beta_{24} \ge 1$, $\beta_{23}+\beta_{34} \ge 1$ and  $\beta_{24} +\beta_{34} \ge 1$.
   This means that at least two of the indices $\beta_{23}$, $\beta_{24}$ and $\beta_{34}$ is larger or equal to  $1$. Considering the worst case, we can assume that
   $\beta_{23}=1$ and $\beta_{34}=1$. This leads to
   \begin{equation}  \label{ma1}
   A^{(0)}_{n,\beta} \le  n^{-1}      \sum_{ |k_1|, |k_2|, |k_3| \le n}
   | \rho(k_1) \rho(k_2)  \rho(k_3) \rho(k_2-k_1) \rho(k_3-k_2)|  .
   \end{equation}
   Using (\ref{equ6c}) and H\"older's inequality we obtain
    \[
   A^{(0)}_{n,\beta} \le  C n^{-1}  \sum_{|k| \leq n} |\rho(k)|  \le  C n^{-\frac 23}  \left( \sum_{|k| \leq n} |\rho(k)|^{\frac 32} \right)^{\frac 23}.
   \]
 
    \medskip
   \noindent
   {\it Case 2:}  Suppose that  the graph is  connected. This means that there is an edge connecting the sets $V_1$ and $V_2$. Suppose that
   $\beta_{\alpha_0\delta_0} \ge 1$, where $\alpha_0\in \{1,2,3,4\}$ and $\delta_0\in \{5,6,7,8\}$.   We have then $7$ nonzero coefficients $\beta$:   $ \beta_{13}$, $\beta_{13}$, $\beta_{14} $, $\beta_{56}$, $\beta_{57}$, $\beta_{58}$ and $\beta_{\alpha_0\delta_0}$.
   Because all the edges have at least degree $2$, there must be  another  nonzero coefficient $\beta$. Assume it is $\beta_{\alpha_1\delta_1}$.
    Then, the worse case will be when $ \beta_{12} =\beta_{13} = \beta_{14} =\beta_{56} =\beta_{57}= \beta_{58}=\beta_{\alpha_0\delta_0}=\beta_{\alpha_1\delta_1}=1$ and all the other coefficients are zero. Consider the change of variables $i_1-i_2=k_1$, $i_1-i_3 =k_2$, $i_1-i_4=k_3$, $i_5-i_6= k_4$, $i_5-i_7 =k_5$,
    $i_5-i_8=k_6$, $i_{\alpha_0}-i_{\delta_0}= k_7$. Then, it is easy to show that
    $i_{\alpha_1} - i_{\delta_1} = {\bf k} \cdot {\bf v}$, where ${\bf k} = (k_1,\dots, k_5)$ and ${\bf v}$ is a $7$-dimensional vector whose components are  $0$, $1$ or $-1$.   Applying  (\ref{equ6b}) and  H\"older's inequality yields
       \[
   A^{(1)}_{n,\beta} \le   
  C n^{-2}  \left( \sum_{|k| \leq n} |\rho(k)| \right)^{6}    \le  C n^{-2}   \left( \sum_{|k| \leq n} |\rho(k)| ^{\frac 32}\right)^{4}.
   \]
   This completes the proof of (\ref{gr1}).

\medskip
\noindent
{\it Proof of (\ref{gr2}):} \quad  We have
\begin{align*}
\sum_{q_1,\dots,q_8=2}^{\infty}  
 \sum_{\beta \in \mathcal{C}^{(1)}_q}  L^{(1)}_{q,\beta} &= \E \left( \, \left|
 (A(g''')^{(N)})(X_1)  (A(g_1)^{(N)} (X_1))^3  \right.\right. \\
 & \quad \left.\left. +(A(g')^{(N)})(X_1)(A(g'')^{(N)})(X_1) (A(g_1)^{(N)} (X_1))^2 \right|^2 \right).
 \end{align*}
 Applying H\"older's inequality, yields
 \begin{align*}
 \sum_{q_1,\dots,q_8=2}^{\infty}  
 \sum_{\beta \in \mathcal{C}^{(1)}_q}  L^{(1)}_{q,\beta}& \le 2   \| A(g''')^{(N)}\|_{L^8(\R ,\gamma)}^2  \| A(g_1)^{(N)}\|_{L^8(\R ,\gamma)}^6 \\
 & \quad  + 2 \| A(g')^{(N)}\|_{L^8(\R ,\gamma)}^2\| A(g'')^{(N)}\|_{L^8(\R ,\gamma)}^2\| A(g_1)^{(N)}\|_{L^8(\R ,\gamma)}^4.
\end{align*}
By Proposition  \ref{truncated} and our hypothesis, taking the limit as $N$ tends to infinity,  it follows that 
 \begin{align*}
 \sum_{q_1,\dots,q_8=2}^{\infty}  
 \sum_{\beta \in \mathcal{C}^{(1)}_q}  L^{(1)}_{q,\beta}& \le 2   \| A(g''')\|_{L^8(\R ,\gamma)}^2  \| A(g_1)\|_{L^8(\R ,\gamma)}^6 \\
 & \quad  + 2 \| A(g')\|_{L^8(\R ,\gamma)}^2\| A(g'')\|_{L^8(\R ,\gamma)}^2\| A(g_1)\|_{L^8(\R ,\gamma)}^4 <\infty.
\end{align*}

\medskip
\noindent
{\it Case $i=2$}. \quad 
 For  $\mathbb{E}[ |\Phi^{(2)}_{n, N}|^2]$ we have
\begin{align*}
\E\left( (\Phi^{(2)} _{n,N})^2  \right) &= \frac 1{n^4} \sum_{i_1, \dots, i_8=1}^n \sum_{q_1, \dots, q_8 =2}^N
M^{(2)}_q \E \left( H_{q_1-2}(X_{i_1})H_{q_2-2}(X_{i_2}) 
    H_{q_3-1}(X_{i_3}) H_{q_4-1}(X_{i_4}) \right. \\
    &  \left.\quad\times  H_{q_5-2}(X_{i_5})H_{q_6-2}(X_{i_6}) 
    H_{q_7-1}(X_{i_7}) H_{q_8-1}(X_{i_8})\right)\\
    & \quad  \times     \rho(i_1-i_2)\rho(i_1-i_3)\rho(i_2-i_4)   \rho(i_5-i_6)\rho(i_5-i_7)\rho(i_6-i_8),
\end{align*}
where
\[
M^{(2)}_q=  \left( \prod_{j=1}^8 c_{q_j}\right)   9q_1(q_1-1) (q_2-1) q_5(q_5-1)  (q_6-1)).
\]
This yields
\begin{align*}
\E\left( (\Phi^{(2)} _{n,N})^2  \right) & \leq \frac{1}{n^{4}} \sum_{i_1, \dots ,i_8=1}^n \sum_{q_1, \dots q_8=2}^{N}     \sum_{\beta \in \mathcal{D}_q^{(2)}} K_{q,\beta}^{(2)} \left(\prod_{1 \leq k< l \leq 8} |\rho(i_k-i_l)|^{\beta_{kl}}\right)  \\
&\quad \times  |\rho(i_1-i_2)\rho(i_1-i_3)\rho(i_2-i_4)\rho(i_5-i_6)\rho(i_5-i_7)\rho(i_6-i_8) |,
\end{align*}
where 
\[
K_{q,\beta}^{(2)}=  \frac{ 9q_1 q_5 
   \prod_{j=1 }^8 | c_{q_j} | (q_j-1)! }{\prod_{1\leq k <l\leq 8} \beta_{kl}!}
\]
and 
\begin{align*}
 \mathcal{D}_q^{(2)}& =\{(\beta_{kl})_{1\leq k <l \leq 8}: \sum_{k \textit{ or } l =j} \beta_{kl}= q_j-1 \text{ for } j= 3,4,7,8 \\
 &\qquad \text{ and } \sum_{k \textit{ or } l =j} \beta_{kl}= q_j-2 \text{ for } j= 1,2,5,6\}.
\end{align*}
Changing the exponents $\beta_{jk}+1$ in to $\beta_{jk}$ for $(j,k) \in \{ (1,2), (1,3), (2,4), (5,6), (5,7), (6,8)\}$, we can write
\[
\E\left( (\Phi^{(2)} _{n,N})^2  \right)  \leq \frac{1}{n^{4}} \sum_{i_1, \dots ,i_8=1}^n \sum_{q_1, \dots q_8=2}^{N}     \sum_{\beta \in \mathcal{C}_q^{(2)}} L_{q,\beta}^{(2)} \left(\prod_{1 \leq k< l \leq 8} |\rho(i_k-i_l)|^{\beta_{kl}}\right),
\]
where 
\[
L_{q,\beta}^{(2)}= \frac{ 9q_1q_5 \prod _{j=1}^8  | c_{q_j} |(q_j-1)!}{(\beta_{12}-1)!
(\beta_{13}-1)!(\beta_{24}-1)!(\beta_{56}-1)!(\beta_{57}-1)(\beta_{68}-1)!\prod_{(k,l) \in \mathcal{E}} \beta_{kl}!}\ ,
\]
with $\mathcal{E} = \{ (k,l): 1\le k<l \le 8,  (k,l) \not= (1,2), (1,3), (2,4), (5,6), (5,7), (6,8)\}$ 
and
\[
    \mathcal{C}_q^{(2)}=\{(\beta_{kl})_{1\leq k <l \leq 8}: \sum_{k \textit{ or } l =j} \beta_{kl}= q_j \text{ for } j= 1,\dots,8 \text{ and }  \beta_{12}, \, \beta_{13}, \, \beta_{24}, \beta_{56}, \beta_{57}, \beta_{6,8} \geq 1\}.
 \]
Then, we have
\[
\E\left (\Phi^{(2)} _{n,N})^2  \right) \le   \sup_{\beta\in \mathcal{C}^{(12)}_q}  A^{(2)}_{n,\beta}   \sum_{q_1,\dots,q_8=2}^{N}  
 \sum_{\beta \in \mathcal{C}^{(2)}_q} 
|L^{(2)}_{q,\beta}|,
 \]
where 
\[
A^{(2)}_{n,\beta}  = \frac{1}{n^4} \sum_{i_1, \dots, i_8=1}^n \prod_{1 \leq j \leq k \leq 8}|\rho(i_i-i_k)|^{\beta_{jk}}
\]
and the supremum is taken over all sets of  nonnegative integers $\beta_{jk}$, $ 1\le j<k \le 8$,   satisfying
 $\beta_{12}, \, \beta_{13}, \, \beta_{24}, \beta_{56}, \beta_{57}, \beta_{68} \geq 1$,
 $\beta_{jk}\le 2$ for $  1\le j<k \le 8$ and
$$
2 \le  \sum_{ j \, {\rm or} \, k =\ell} \beta_{jk}, \quad {\rm for}  \quad 1\le \ell \le 8.
$$

We need to estimate  $A^{(2)}_{n,\beta}$ and to show that 
\begin{equation}  \label{a4}
 \sum_{q_1,\dots,q_8=2}^{\infty}  
 \sum_{\beta \in \mathcal{C}^{(2)}_q} 
L_{q,\beta}^{(2)} <\infty.
\end{equation}

\medskip
\noindent
{\it Estimation of $A^{(2)}_{n,\beta}$:} \quad 
We claim that  
 $$
\sup_{\beta} A^{(2)}_{n,\beta} \le  C n^{-1} \left(\sum_{|k| \leq n} |\rho(k)| ^{\frac 32}\right)^4.
 $$
As in the proof of Theorem \ref{thm2}, we will make use of ideas from graph theory. The exponents $\beta_{jk}$ induce an unordered simple graph on the set of vertices  $V=\{ 1,2,3,4,5,8\}$ by putting an edge between $j$ and $k$ whenever $\beta_{jk }\not=0$. Because  $\beta_{12}\ge 1$, $\beta_{13}\ge 1$,  $\beta_{24}\ge 1$, $\beta_{56}\ge 1, \beta_{57}\ge 1$ and  $\beta_{68}\ge 1$, there  are edges connecting the pairs of vertices $(1,2)$, $(1,3)$, $(2,4)$, $(5,6)$
$(5,7)$ and $(6,8)$. Condition (\ref{bar4}) means that the degree of each vertex is at least $2$.
   Then we consider two cases, depending whether graph is connected or not.
   
   \medskip
   \noindent
   {\it Case 1:} Suppose that  the graph is not connected.  This means that
    $\beta_{jk} =0$ if $j \in \{1,2,3,4\}$ and $k\in \{5,6,7,8\}$ and there is no edge between the sets
   $V_1=\{ 1,2,3,4\}$ and $V_2=\{ 5,6,7,8\}$. Therefore,
    \[
    A^{(2)}_{n,\beta} \le     (A^{(0)}_{n,\beta} )^2,
    \]
    where
     \[
A^{(0)}_{n,\beta}  = \frac{1}{n^2} \sum_{i_1, \dots, i_4=1}^n \prod_{1 \leq j \leq k \leq 4}|\rho(i_i-i_k)|^{\beta_{jk}}
\]
and the nonnegative integers $\beta_{jk}$, $ 1\le j<k \le 4$,   satisfy
 $\beta_{12}, \, \beta_{13}, \, \beta_{24} \geq 1$,
 $\beta_{jk}\le 2$ for $  1\le j<k \le 4$ and
$$
2 \le  \sum_{ j \, {\rm or} \, k =\ell} \beta_{jk}, \quad {\rm for}  \quad 1\le \ell \le 4.
$$
     As a consequence,
    $\beta_{23}+ \beta_{34} \ge 1$ and  $\beta_{14}+\beta_{34} \ge 1$.
   This means $\beta_{34} \ge 1$ or both $ \beta_{23}$ and $\beta_{14}$ are  larger or equal than one.  There are two possible cases:
   \begin{itemize}
   \item[(i)] Suppose  $\beta_{34} \ge 1$,
     Considering the worst case, we can assume that
 $\beta_{34}=1$. Then,  applying  (\ref{equ6a}) and H\"older's inequality, we obtain
   \[
   A^{(0)}_{n,\beta} \le  n^{-1}      \sum_{ |k_1|, |k_2|, |k_3| \le n}
   | \rho(k_1) \rho(k_2)  \rho(k_3) \rho(k_1+k_3-k_2)  |   \le  n^{-1} \left(\sum_{|k| \leq n} |\rho(k)| ^{\frac 43}\right)^3.
   \] 
   By H\"older's inequality, we can show that
   \[
     ( A^{(0)}_{n,\beta})^2 \le C n^{-1} \left(\sum_{|k| \leq n} |\rho(k)| ^{\frac 32}\right)^4.
     \]
   \item[(ii)] Suppose $ \beta_{23}\ge 1$ and $\beta_{14}\ge 1$. Then,
      \[
   A^{(0)}_{n,\beta} \le  n^{-1}      \sum_{ |k_1|, |k_2|, |k_3| \le n}
   | \rho(k_1) \rho(k_2)  \rho(k_3) \rho(k_1+k_3)  \rho(k_1-k_2) |,
   \] 
   and this case can be treated as (\ref{ma1}).
   \end{itemize}
 
    \medskip
   \noindent
   {\it Case 2:}  Suppose that  the graph is  connected. This means that there is an edge connecting the sets $V_1$ and $V_2$. Suppose that
   $\beta_{\alpha_0\delta_0} \ge 1$, where $\alpha_0\in \{1,2,3,4\}$ and $\delta_0\in \{5,6,7,8\}$.   We have then $7$ nonzero coefficients $\beta$:   $ \beta_{12}$, $\beta_{13}$, $\beta_{24} $, $\beta_{56}$, $\beta_{57}$, $\beta_{68}$ and $\beta_{\alpha_0\delta_0}$.
   Because all the edges have at least degree $2$, there must be  another  nonzero coefficient $\beta$. Assume it is $\beta_{\alpha_1\delta_1}$.
    Then, the worse case will be when $ \beta_{12} =\beta_{13} = \beta_{24} =\beta_{56} =\beta_{57}= \beta_{68}=\beta_{\alpha_0\delta_0}=\beta_{\alpha_1\delta_1}=1$ and all the other coefficients are zero. Consider the change of variables $i_1-i_2=k_1$, $i_1-i_3 =k_2$, $i_2-i_4=k_3$, $i_5-i_6= k_4$, $i_5-i_7 =k_5$,
    $i_6-i_8=k_6$, $i_{\alpha_0}-i_{\delta_0}= k_7$. Then, it is easy to show that
    $i_{\alpha_1} - i_{\delta_1} = {\bf k} \cdot {\bf v}$, where ${\bf k} = (k_1,\dots, k_5)$ and ${\bf v}$ is a $7$-dimensional vector whose components are  $0$, $1$ or $-1$.   Then, using (\ref{equ6b})   and H\"older's inequality, we obtain
       \[
   A^{(1)}_{n,\beta} \le   
  C n^{-2}  \left( \sum_{|k| \leq n} |\rho(k)| \right)^{6}    \le  C n^{-2}   \left( \sum_{|k| \leq n} |\rho(k)| ^{\frac 32}\right)^{4}.
   \]

\medskip
\noindent
{\it Proof of (\ref{a4}):} \quad  We have
\begin{align*}
\sum_{q_1,\dots,q_8=2}^{\infty}  
 \sum_{\beta \in \mathcal{C}^{(2)}_q}L^{(2)}_{q,\beta} &=9 \E\left( \left|
A(g'')^{(N)}(X_1) A(g_1')(X_1) A(g_1)(X_1)^2
   \right|^2 \right)\\
   &\le  9 \| A(g'')^{(N)} \| _{L^8(\R,\gamma)} ^2\| A(g'_1)^{(N)} \| _{L^8(\R,\gamma)} ^2\| A(g_1)^{(N)} \| _{L^8(\R,\gamma)} ^4,
    \end{align*}
    which converges as $N\rightarrow \infty $ to
    \[
    9 \| A(g'') \| _{L^8(\R,\gamma)} ^2\| A(g'_1) \| _{L^8(\R,\gamma)} ^2\| A(g_1) \| _{L^8(\R,\gamma)} ^4 <\infty.
    \]

\medskip
\noindent
{\it Case $i=3$.}
The term  $\mathbb{E}[ |\Phi^{(3)}_{n, N}|^2]$   can be handled in a similar way and we omit the details. 
\end{proof}

\end{document}